\documentclass{amsart}
\usepackage{amsmath,amsthm,amssymb,amsfonts,amscd, array}

\newtheorem{Theorem}{Theorem}[section]
\newtheorem{defn}[Theorem]{Definition}

\newtheorem{prp}[Theorem]{Proposition}
\newtheorem{lemma}[Theorem]{Lemma}
\newtheorem{Example}[Theorem]{Example}
\newtheorem{rmk}[Theorem]{Remark}
\title[On Complex Intuitionistic Fuzzy Lie Sub-superalgebras]
{On Complex Intuitionistic Fuzzy Lie Sub-superalgebras}
\author[Ameer Jaber]{}
 \email{ameerj@hu.edu.jo}
 \medskip

\begin{document}

\maketitle

\centerline{\scshape  Ameer Jaber}
\medskip

{\footnotesize \centerline{ Department of Mathematics }
\centerline{ The Hashemite University} \centerline{ Zarqa 13115, Jordan
} }
\begin{abstract}
A complex intuitionistic fuzzy Lie superalgebra is a generalization of intuitionistic fuzzy Lie superalgebra whose membership function takes values in the unit disk in the complex plane. In this research, the concepts of complex intuitionistic fuzzy sets are introduced to Lie superalgebras. We define the complex intuitionistic fuzzy Lie sub-superalgebras and complex intuitionistic fuzzy ideals of Lie superalgebras. Then, we study some related properties of complex intuitionistic fuzzy Lie sub-superalgebras and complex intuitionistic fuzzy ideals. Finally, we define the image and preimage of complex intuitionistic fuzzy Lie sub-superalgebra under Lie superalgebra anti-homomorphism. The properties of anti-complex intuitionistic fuzzy Lie sub-superalgebras and anti-complex intuitionistic fuzzy ideals under anti-homomorphisms of Lie superalgebras are investigated.
\end{abstract}
\section{Introduction}
The notion of intuitionistic fuzzy sets was introduced by Atanassov (see \cite{A0}). He presented in \cite{A0} the idea of intuitionistic fuzzy sets. He also in \cite{A} defined some properties of intuitionistic fuzzy sets. Atanassov presented in \cite{A1} interesting new operations about intuitionistic fuzzy sets. An intuitionistic fuzzy set is the generalization of fuzzy set. Recently, Biswas applied the concepts of intuitionistic fuzzy sets to the theory of groups and studied intuitionistic fuzzy subgroups of a group (see \cite{RB}); also Banerjee studied intuitionistic fuzzy subrings and ideals of a ring (see \cite{BB}). Moreover, Jun investigated the concept of intuitionistic nil-radicals of intuitionistic fuzzy ideals in rings (see \cite{YB}) and Davvaz, Dudek and Jun applied the notion of intuitionistic fuzzy sets to certain types of modules (see \cite{BD}). Then in \cite{CZ} W. chen and S. Zhang introduced the concept of intuitionistic fuzzy Lie superalgebras and intuitionistic fuzzy ideals. It is known that fuzzy sets are intuitionistic fuzzy sets but the converse is not necessarily true (for more details see \cite{A}). More recently, Alkouri and Salleh gave in \cite{AS} the idea of complex intuitionistic fuzzy subsets and then they enlarge the basic properties of it. This concept became more effective and useful in scientific field because it deals with degree of membership and nonmembership in complex plane. They also initiated the concept of complex intuitionistic fuzzy relation and developed fundamental operation of complex intuitionistic fuzzy sets in \cite{AS1, AS2}. Then Garg and Rani made in \cite{GR} a huge effort to generalize the notion of complex intuitionistic fuzzy sets in decision-making problems.\\
In \cite{SS-1}, S. Shaqaqha introduced the concepts of complex fuzzy sets to the theory of Lie algebras and studied complex fuzzy Lie subalgebras. Furthermore, in \cite{SS-2, SS-6}, S. Shaqaqha and M. Al-Deiakeh  introduced the concepts of complex intuitionistic fuzzy Lie algebras and complex intuitionistic fuzzy Lie ideals and they study the relation between complex intuitionistic fuzzy Lie subalgebras (ideals) and intuitionistic fuzzy Lie subalgebras (ideals). Also, in \cite{SS-3}, S. Shaqaqha characterized Noetherian and Artinian Gamma rings by complex fuzzy ideals. Moreover, in \cite{SS-4}, he introduced the notion of intuitionistic fuzzy Lie subalgebras and intutionistic fuzzy Lie ideals of $n$-Lie algebras which is a generalization of intuitionistic fuzzy Lie algebras. More recently, in \cite{SS-5}, he introduced the concept of complex fuzzy $\Gamma$-rings and he showed that there are isomorphism theorems concerning complex fuzzy $\Gamma$-rings as for rings.\\
I must point out here that the main idea of this article is to introduce the concepts of complex intuitionistic fuzzy Lie superalgebras and complex intuitionistic fuzzy ideals, which are generalizations of intuitionistic fuzzy Lie superalgebras and intuitionistic fuzzy ideals applied by W. chen and S. Zhang in \cite{CZ}. We prepared this paper as follows: In section 2, we recall some basic definitions and notions which will be used in what follows. In section 3, we introduce the definition of $\mathbb{Z}_2$-graded complex intuitionistic fuzzy  vector subspace, define complex intuitionistic fuzzy Lie sub-superalgebras, complex intuitionistic fuzzy ideals and consider their characterization. Finally, in section 4, we discuss images and preimages of complex intuitionistic fuzzy Lie sub-superalgebras and complex intuitionistic fuzzy ideals under anti-homomorphisms.
\section{Complex intuitionistic fuzzy sets}
Let $X\not=\phi$. A complex intuitionistic fuzzy set on $X$ is an object having the form $A=\{(x,\lambda_A(x),\rho_A(x))\ | x\in X\}$, where the complex functions $\lambda_A : X\rightarrow \mathbb{C}$ and $\rho_A : X\rightarrow \mathbb{C}$ denote the degree of membership (namely $\lambda_A(x)$) and the degree of non-membership (namely $\rho_A(x)$) of each element $x\in X$ to the set $A$, respectively, that assign to any element $x\in X$ complex numbers $\lambda_A(x)$, $\rho_A(x)$ lie within the unit circle with the property $|\lambda_A(x)|+|\rho_A(x)|\leq 1$.
For the sake of simplicity, we shall use the symbol $A=(\lambda_A, \rho_A)$ for the complex intuitionistic fuzzy set $A=\{(x,\lambda_A(x),\rho_A(x))\ | x\in X\}$.\\
We shall assume $\lambda_A(x)$, $\rho_A(x)$ will be represented by $r_A(x)e^{i2\pi \omega_A(x)}$ and $\hat{r}_A(x)e^{i2\pi \hat{\omega}_A(x)}$, respectively, where $i=\sqrt{-1} $, $r_A(x), \hat{r}_A(x), \omega_A(x), \hat{\omega}_A(x) \in [0,1]$. Thus the property of $|\lambda_A(x)|+|\rho_A(x)|\leq 1$ implies $r_A(x)+\hat{r}_A(x)\leq1$. Note that the intuitionistic fuzzy set is a special case of complex intuitionistic fuzzy set with $\omega_A(x)=\hat{\omega}_A(x)=0.$
Also, if $\rho_A(x)=(1-r_A(x))e^{i2\pi(1-\omega_A(x))}$, then we obtain a complex fuzzy set.
Let $\alpha e^{i2\pi\beta}$ and $\gamma e^{i2\pi\delta}$ be two complex numbers, where $\alpha, \beta, \gamma, \delta\in [0,1]$. By $\alpha e^{i2\pi\beta}\leq \gamma e^{i2\pi\delta}$ we mean $\alpha\leq\gamma$ and $\beta\leq\delta$.
In this paper,we use the symbols $a\wedge b=min\{a,b\}$ and $a\vee b=max\{a,b\}$.
Let $A=(\lambda_A,\rho_A)$ be a complex intuitionistic fuzzy set on $X$ with the degree of membership $\lambda_A(x)=r_A(x)e^{i2\pi \omega_A(x)}$ and the degree of non-membership $\rho_A(x)=\hat{r}_A(x)e^{i2\pi\hat{\omega}_A(x)}$. Then $A$ is said to be a homogeneous complex intuitionistic fuzzy set if the following two conditions hold $\forall x,y\in X$\\
(1) $r_A(x)\leq r_A(y)$ if and only if $\omega_A(x)\leq\omega_A(y)$,\\
(2) $\hat{r}_A(x)\leq \hat{r}_A(y)$ if and only if $\hat{\omega}_A(x)\leq\hat{\omega}_A(y)$.\\
Let $A=(\lambda_A,\rho_A)$ and $B=(\lambda_B,\rho_B)$ be two complex intuitionistic fuzzy sets on the same set $X$, we say that $A$ is homogeneous with $B$ if the following conditions hold $\forall x,y\in X$\\
(1) $r_A(x)\leq r_B(y)$ if and only if $\omega_A(x)\leq\omega_B(y)$,\\
(2) $\hat{r}_A(x)\leq \hat{r}_B(y)$ if and only if $\hat{\omega}_A(x)\leq\hat{\omega}_B(y)$.\\
\begin{defn}
Let $K$ be any field, and let $V$ be a $K$-vector space. A complex intuitionistic fuzzy (CIF for short) set on $V$ defined as an object having the form $A=\{(x,\lambda_A(x),\rho_A(x))\ | x\in V\}$, where the complex functions $\lambda_A : V\rightarrow \mathbb{C}$ and $\rho_A : V\rightarrow \mathbb{C}$ denote the degree of membership (namely $\lambda_A(x)$) and the degree of non-membership (namely $\rho_A(x)$) of each element $x\in V$ to the set $A$, respectively, that assign to any element $x\in V$ complex numbers $\lambda_A(x)$, $\rho_A(x)$ lie within the unit circle with the property $|\lambda_A(x)|+|\rho_A(x)|\leq 1$.
\end{defn}
We shall use the symbol $A=(\lambda_A, \rho_A)$ for the CIF set $A=\{(x,\lambda_A(x),\rho_A(x))\ | x\in X\}$.
\begin{defn}
Let $A=(\lambda_A,\rho_A)$ and $B=(\lambda_B,\rho_B)$ be CIF vector subspaces of a vector subspace $V$. Then\\
(1) $A\subseteq B$ if $\lambda_A(x)\leq\lambda_B(x)$ and $\rho_A(x)\geq\rho_B(x)$,\\
(2) $A\cap B=\{x,\lambda_A(x)\wedge\lambda_B(x),\rho_A(x)\vee\rho_B(x)|x\in V\}$,\\
(3) $A\cup B=\{x,\lambda_A(x)\vee\lambda_B(x),\rho_A(x)\wedge\rho_B(x)|x\in V\}$\\
\end{defn}
\begin{defn}
Let $V$ be a $K$-vector space. A CIF set $A=(\lambda_A, \rho_A)$ of a vector space $V$ is called a CIF vector space of $V$, if it satisfies the following conditions\\
for any $x,y\in V$, $\alpha\in K$\\
(1) $\lambda_A(x+y)\geq \lambda_A(x)\wedge\lambda_A(y)$, and $\rho_A(x+y)\leq \rho_A(x)\vee\rho_A(y)$\\
(2) $\lambda_A(\alpha x)\geq \lambda_A(x)$, and $\rho_A(\alpha x)\leq \rho_A(x)$.
\end{defn}
From this definition, we know that for any $x\in V$, $\lambda_A(0)\geq\lambda_A(x)$ and $\rho_A(0)\leq\rho_A(x)$. In this paper, we always assume that $\lambda_A(0)=1e^{i2\pi}=1$ and $\rho_A(0)=0e^{i2(\pi)0}=0$.
\begin{defn}\label{def-1}
Let $A=(\lambda_A,\rho_A)$ and $B=(\lambda_B,\rho_B)$ be CIF vector subspaces of a vector subspace $V$, where $\lambda_A=r_Ae^{i2\pi\omega_A},\lambda_B=r_Be^{i2\pi\omega_B}$and $\rho_A=\hat{r}_Ae^{i2\pi\hat{\omega}_A}, \rho_B=\hat{r}_Be^{i2\pi\hat{\omega}_B}$. If $A$ is homogenous with $B$. Then the complex intuitionistic sum of $A=(\lambda_A,\rho_A)$ and $B=(\lambda_B,\rho_B)$ is defined to the CIF set $A+B=(\lambda_{A+B},\rho_{A+B})$ of $V$ given by
$$\lambda_{A+B}(x)=\begin{cases} \sup_{x=a+b}\{(r_A(a)\wedge r_B(b))\}e^{i2\pi\sup_{x=a+b}\{ (\omega_A(a)\wedge\omega_B(b))\}} & :\quad {\rm if}\ x=a+b\\
0 & :\quad otherwise,\end{cases}$$
$$\rho_{A+B}(x)=\begin{cases} \inf_{x=a+b}\{(\hat{r}_A(a)\vee\hat{r}_B(b))\}e^{i2\pi\inf_{x=a+b}\{ (\hat{\omega}_A(a)\vee\hat{\omega}_B(b))\}} & :\quad {\rm if}\ x=a+b\\
1 & :\quad otherwise,\end{cases}$$
\end{defn}
Further, if $A\cap B=(\lambda_{A\cap B},\rho_{A\cap B})$, where
$$\lambda_{A\cap B}(x)=\begin{cases} 0 & :\quad x\not=0\\
1 & :\quad x=0,\end{cases}\ \ \ {\rm and}\ \ \
\rho_{A\cap B}(x)=\begin{cases} 1 & :\quad x\not=0\\
0 & :\quad x=0,\end{cases}.$$ Then $A+B$ is said to be the direct sum and denoted by $A\oplus B$.
\begin{lemma}\label{mylemma-1}\textnormal{
Let $A=(\lambda_A,\rho_A)$ and $B=(\lambda_B,\rho_B)$ be CIF vector subspaces of a vector subspace $V$ such that $A$ is homogenous with $B$. Then $A+B=(\lambda_{A+B},\rho_{A+B})$ is also a CIF vector subspaces of a vector subspace $V$.}
\end{lemma}
\begin{proof}
One can easily show that $A+B$ is homogenous. Now, for any $x,y\in V$, we have
\begin{eqnarray*}
\lambda_{A+B}(x)\wedge\lambda_{A+B}(y)&=&r_{A+B}(x)e^{i2\pi\omega_{A+B}(x)}\wedge r_{A+B}(y)e^{i2\pi\omega_{A+B}(y)}\\
&=&(r_{A+B}(x)\wedge r_{A+B}(y))e^{i2\pi(\omega_{A+B}(x)\wedge\omega_{A+B}(x))},
\end{eqnarray*}
and
\begin{eqnarray*}
r_{A+B}(x)\wedge r_{A+B}(y)&=&\sup_{x=a+b}\{r_A(a)\wedge r_B(b)\}\wedge\sup_{y=c+d}\{r_A(c)\wedge r_B(d)\}\\
&=&\sup_{x=a+b,y=c+d}\{(r_A(a)\wedge r_B(b))\wedge(r_A(c)\wedge r_B(d))\}\\
&=&\sup_{x=a+b,y=c+d}\{(r_A(a)\wedge r_A(c))\wedge(r_B(b)\wedge r_B(d))\}\\
&\leq&\sup_{x=a+b,y=c+d}\{r_A(a+c)\wedge r_B(b+d)\}\\
&=&r_{A+B}(x+y).
\end{eqnarray*}
Similarly, $\omega_{A+B}(x)\wedge\omega_{A+B}(y)\leq\omega_{A+B}(x+y)$. Thus,
\begin{eqnarray*}
\lambda_{A+B}(x)\wedge\lambda_{A+B}(y)&=&(r_{A+B}(x)\wedge r_{A+B}(y))e^{i2\pi(\omega_{A+B}(x)\wedge\omega_{A+B}(x))}\\
&\geq&r_{A+B}(x+y)e^{i2\pi\omega_{A+B}(x+y)}=\lambda_{A+B}(x+y),
\end{eqnarray*}
and
\begin{eqnarray*}
\rho_{A+B}(x)\vee\rho_{A+B}(y)&=&\hat{r}_{A+B}(x)e^{i2\pi\hat{\omega}_{A+B}(x)}\vee \hat{r}_{A+B}(y)e^{i2\pi\hat{\omega}_{A+B}(y)}\\
&=&(\hat{r}_{A+B}(x)\vee\hat{r}_{A+B}(y))e^{i2\pi(\hat{\omega}_{A+B}(x)\vee\hat{\omega}_{A+B}(x))},
\end{eqnarray*}
and
\begin{eqnarray*}
\hat{r}_{A+B}(x)\vee\hat{r}_{A+B}(y)&=&\inf_{x=a+b}\{\hat{r}_A(a)\vee \hat{r}_B(b)\}\vee\inf_{y=c+d}\{\hat{r}_A(c)\vee\hat{r}_B(d)\}\\
&=&\inf_{x=a+b,y=c+d}\{(\hat{r}_A(a)\vee\hat{r}_B(b))\vee(\hat{r}_A(c)\vee\hat{r}_B(d))\}\\
&=&\inf_{x=a+b,y=c+d}\{(\hat{r}_A(a)\vee\hat{r}_A(c))\vee(\hat{r}_B(b)\vee\hat{r}_B(d))\}\\
&\geq&\inf_{x=a+b,y=c+d}\{\hat{r}_A(a+c)\vee\hat{r}_B(b+d)\}\\
&=&\hat{r}_{A+B}(x+y).
\end{eqnarray*}
Similarly, $\hat{\omega}_{A+B}(x)\vee\hat{\omega}_{A+B}(y)\geq\hat{\omega}_{A+B}(x+y)$. Thus,
\begin{eqnarray*}
\rho_{A+B}(x)\vee\rho_{A+B}(y)&=&(\hat{r}_{A+B}(x)\vee \hat{r}_{A+B}(y))e^{i2\pi(\hat{\omega}_{A+B}(x)\vee\hat{\omega}_{A+B}(x))}\\
&\geq&\hat{r}_{A+B}(x+y)e^{i2\pi\hat{\omega}_{A+B}(x+y)}=\rho_{A+B}(x+y).
\end{eqnarray*}
Also, for $\alpha\in K$, we have
$\lambda_{A+B}(x)=r_{A+B}(x)e^{i2\pi\omega_{A+B}(x)}$, and
\begin{eqnarray*}
r_{A+B}(x)&=&\sup_{x=a+b}\{r_A(a)\wedge r_B(b)\}\\
&\leq&\sup_{\alpha x=\alpha a+\alpha b}\{r_A(\alpha a)\wedge r_B(\alpha b)\}\\
&\leq&\sup_{\alpha x=c+d}\{r_A(c)\wedge r_B(d)\}\\
&=&r_{A+B}(\alpha x).
\end{eqnarray*}
Similarly, $\omega_{A+B}(x)\leq\omega_{A+B}(\alpha x)$. Thus,
$$\lambda_{A+B}(x)=r_{A+B}(x)e^{i2\pi\omega_{A+B}(x)}\leq r_{A+B}(\alpha x)e^{i2\pi\omega_{A+B}(\alpha x)}=\lambda_{A+B}(\alpha x),$$ and
$\rho_{A+B}(x)=\hat{r}_{A+B}(x)e^{i2\pi\hat{\omega}_{A+B}(x)}$, but
\begin{eqnarray*}
\hat{r}_{A+B}(x)&=&\inf_{x=a+b}\{\hat{r}_A(a)\vee\hat{r}_B(b)\}\\
&\geq&\inf_{\alpha x=\alpha a+\alpha b}\{\hat{r}_A(\alpha a)\vee\hat{r}_B(\alpha b)\}\\
&\geq&\inf_{\alpha x=c+d}\{\hat{r}_A(c)\vee\hat{r}_B(d)\}\\
&=&\hat{r}_{A+B}(\alpha x).
\end{eqnarray*}
Similarly, $\hat{\omega}_{A+B}(x)\geq\hat{\omega}_{A+B}(\alpha x)$. Thus,
$$\rho_{A+B}(x)=\hat{r}_{A+B}(x)e^{i2\pi\hat{\omega}_{A+B}(x)}\geq\hat{r}_{A+B}(\alpha x)e^{i2\pi\hat{\omega}_{A+B}(\alpha x)}=\rho_{A+B}(\alpha x).$$
Hence, $A+B=(\lambda_{A+B},\rho_{A+B})$ is a CIF vector subspaces of a vector subspace $V$.
\end{proof}
\begin{defn}
Let $A=(\lambda_A, \rho_A)$ be a CIF vector subspace of a $K$-vector space $V$. For $\alpha\in K$ and $x\in V$, define $\alpha A=(\lambda_{\alpha A}, \rho_{\alpha A})$, where $$\lambda_{\alpha A}(x)=\begin{cases} \lambda_A(\alpha^{-1}x)=r_A(\alpha^{-1}x)e^{i2\pi\omega_A(\alpha^{-1}x)} & :\quad \alpha\not=0\\
1 & :\quad \alpha=0, x=0\\
0 & :\quad \alpha=0, x\not=0\end{cases}$$ and
$$\rho_{\alpha A}(x)=\begin{cases} \rho_A(\alpha^{-1}x)=\hat{r}_A(\alpha^{-1}x)e^{i2\pi\hat{\omega}_A(\alpha^{-1}x)} & :\quad \alpha\not=0\\
0 & :\quad \alpha=0, x=0\\
1 & :\quad \alpha=0, x\not=0\end{cases}$$
\end{defn}
\begin{defn}
Let $V$, $V'$ be $K$-vector spaces and let $f : V\rightarrow V'$ be any map. If $A=(\lambda_A, \rho_A)$, $B=(\lambda_B, \rho_B)$ are CIF vector subspaces of $V$ and $V'$, respectively, then the preimage of $B$ under $f$ is defined to be a CIF set $f^{-1}(B)=(\lambda_{f^{-1}}, \rho_{f^{-1}})$, where $\lambda_{f^{-1}}(x)=\lambda_B(f(x))$ and $\rho_{f^{-1}}(x)=\rho_B(f(x))$ for any $x\in V$ and the image of $A=(\lambda_A, \rho_A)$ under $f$ is defined to be the CIF set $f(A)=(\lambda_{f(A)}, \rho_{f(A)})$ where $$\lambda_{f(A)}(y)=\begin{cases} \underset{x\in f^{-1}(y)}{\sup}\{\lambda_A(x)\}=\underset{x\in f^{-1}(y)}{\sup}\{r_A(x)e^{i2\pi\omega_A(x)}\} & :\quad y\in f(V)\\
0 & :\quad y\not\in f(V)\end{cases}$$ and
$$\rho_{f(A)}(y)=\begin{cases} \underset{x\in f^{-1}(y)}{\inf}\{\rho_A(x)\}=\underset{x\in f^{-1}(y)}{\inf}\{\hat{r}_A(x)e^{i2\pi\hat{\omega}_A(x)}\} & :\quad y\in f(V)\\
1 & :\quad y\not\in f(V)\end{cases}$$
\end{defn}
The following results are easy to get. Here we omit the proofs.
\begin{lemma}\textnormal{ Let $A=(\lambda_A, \rho_A)$, where $\lambda_A=r_Ae^{i2\pi\omega_A}$ and $\rho_A=\hat{r}_Ae^{i2\pi\hat{\omega}_A}$, be a CIF vector subspace of $V$. Then $\alpha A=(\lambda_{\alpha A}, \rho_{\alpha A})$ is also a CIF vector subspace of $V$.}
\end{lemma}
\begin{lemma}\textnormal{ Let $A=(\lambda_A, \rho_A)$, where $\lambda_A=r_Ae^{i2\pi\omega_A}$ and $\rho_A=\hat{r}_Ae^{i2\pi\hat{\omega}_A}$, be a CIF vector subspace of $V'$ and $f : V\rightarrow V'$ be any map. Then $f^{-1}(A)=(\lambda_{f^{-1}(A)}, \rho_{f^{-1}(A)})$ is also a CIF vector subspace of $V$, where $\lambda_{f^{-1}(A)}(x)=\lambda_A(f(x))$ and $\rho_{f^{-1}(A)}(x)=\rho_A(f(x))$ for all $x\in V$.}
\end{lemma}
\begin{lemma}\textnormal{ Let $f : V\rightarrow V'$ be any map. If $A=(\lambda_A, \rho_A)$, where $\lambda_A=r_Ae^{i2\pi\omega_A}$ and $\rho_A=\hat{r}_Ae^{i2\pi\hat{\omega}_A}$, is a CIF vector subspace of $V$, then $f(A)=(\lambda_{f(A)}, \rho_{f(A)})$ is a CIF vector subspace of $V'$.}
\end{lemma}
\begin{lemma}\label{mylemma-2}\textnormal{ Let $A=(\lambda_A, \rho_A)$ and $B=(\lambda_B, \rho_B)$ be CIF vector subspace of $V$ such that $A$ is homogenous with $B$, where $\lambda_A=r_Ae^{i2\pi\omega_A}$, $\rho_A=\hat{r}_Ae^{i2\pi\hat{\omega}_A}$ and $\lambda_B=r_Be^{i2\pi\omega_B}$, $\rho_B=\hat{r}_Be^{i2\pi\hat{\omega}_B}$. Then $A\cap B=(\lambda_{A\cap B}, \rho_{A\cap B})$ is a CIF vector subspace of $V$, where}
\begin{eqnarray}
\lambda_{A\cap B}(x)&=&\lambda_A(x)\wedge\lambda_B(x)\nonumber\\
&=&(r_A(x)\wedge r_B(x))e^{i2\pi(\omega_A(x)\wedge\omega_B(x))}\nonumber
\end{eqnarray}
 and
 \begin{eqnarray}
 \rho_{A\cap B}(x)&=&\rho_A(x)\vee\rho_B(x)\nonumber\\
 &=&(\hat{r}_A(x)\vee\hat{r}_B(x))e^{i2\pi(\hat{\omega}_A(x)\vee\hat{\omega}_B(x))}\nonumber
\end{eqnarray}
\end{lemma}
\begin{defn}
A $\mathbb{Z}_2$-graded vector space $V=V_0+V_1$ possessing the operation called the bilinear bracket product, $$[\ ,\ ] : V\times V\stackrel{bilinear}{\longrightarrow}[x,y]\in V$$ is called a Lie superalgebra, if it satisfies the following conditions\\
(1) $[V_i,V_j]\subseteq V_{i+j}$\\
(2) $[x,y]=-(-1)^{|x||y|}[y,x]$ $\forall x,y\in V_0\cup V_1$\\
(3) $[x,[y,z]]-(-1)^{|x||y|}[y,[x,z]]=[[x,y],z]$
\end{defn}
\section{Complex intuitionistic fuzzy lie sub-superalgebras and ideals}
In this section we assume that $V$ is a Lie superalgebra over a field $K$.
\begin{defn}\label{def-2}
Let $V=V_0+V_1$ be a $\mathbb{Z}_2$-graded vector space. Suppose that $A_0=(\lambda_{A_0}, \rho_{A_0})$ and $A_1=(\lambda_{A_1}, \rho_{A_1})$ are CIF vector subspaces of $V_0$ and $V_1$, respectively. Define $\mathfrak{a}_0=(\lambda_{\mathfrak{a}_0}, \rho_{\mathfrak{a}_0})$
where
$$\lambda_{\mathfrak{a}_0}(x)=\begin{cases} \lambda_{A_0}(x) & :\quad x\in V_0\\
0 & :\quad x\not\in V_0\end{cases}\ \ and\ \
\rho_{\mathfrak{a}_0}(x)=\begin{cases} \rho_{A_0}(x) & :\quad x\in V_0\\
1 & :\quad x\not\in V_0\end{cases}$$ and define $\mathfrak{a}_1=(\lambda_{\mathfrak{a}_1}, \rho_{\mathfrak{a}_1})$
where
$$\lambda_{\mathfrak{a}_1}(x)=\begin{cases} \lambda_{A_1}(x) & :\quad x\in V_1\\
0 & :\quad x\not\in V_1\end{cases}\ \ and\ \
\rho_{\mathfrak{a}_1}(x)=\begin{cases} \rho_{A_1}(x) & :\quad x\in V_1\\
1 & :\quad x\not\in V_1\end{cases}$$
\end{defn}
Then $\mathfrak{a}_0=(\lambda_{\mathfrak{a}_0}, \rho_{\mathfrak{a}_0})$ and $\mathfrak{a}_1=(\lambda_{\mathfrak{a}_1}, \rho_{\mathfrak{a}_1})$ are the CIF vector subspaces of $V$. Moreover, we have $\mathfrak{a}_0\cap\mathfrak{a}_1=(\lambda_{\mathfrak{a}_0\cap\mathfrak{a}_1}, \rho_{\mathfrak{a}_0\cap\mathfrak{a}_1})$, where $$\lambda_{\mathfrak{a}_0\cap \mathfrak{a}_1}(x)=\lambda_{\mathfrak{a}_0}(x)\wedge\lambda_{\mathfrak{a}_1}(x)=\begin{cases} 1 & :\quad x=0\\
0 & :\quad x\not=0,\end{cases}$$ and $$\rho_{\mathfrak{a}_0\cap \mathfrak{a}_1}(x)=\rho_{\mathfrak{a}_0}(x)\vee\rho_{\mathfrak{a}_1}(x)=\begin{cases} 0 & :\quad x=0\\ 1 & :\quad x\not=0.\end{cases}$$ So $\mathfrak{a}_0+\mathfrak{a}_1$ is the direct sum and is denoted by $A_0\oplus A_1$. If $A=(\lambda_A, \rho_A)$ is a CIF vector subspace of $V$ and $A=A_0\oplus A_1$, then $A=(\lambda_A, \rho_A)$ is called a $\mathbb{Z}_2$-graded CIF vector subspaces of $V$.
\begin{rmk}\textnormal{
(1)}\begin{align*}
\lambda_A(x)&=\lambda_{A_0\oplus A_1}(x)\\
&=\sup_{x=\alpha+\beta}\{\lambda_{\mathfrak{a}_0}(\alpha)\wedge\lambda_{\mathfrak{a}_1}(\beta)\}\\
&=\sup_{x=\alpha+\beta}\{\lambda_{\mathfrak{a}_0}(\alpha_0+\alpha_1)\wedge\lambda_{\mathfrak{a}_1}(\beta_0+\beta_1)\},\ where\ \alpha=\alpha_0+\alpha_1,\beta=\beta_0+\beta_1 \\
&=\sup_{x=\alpha+\beta}\{\lambda_{\mathfrak{a}_0}(\alpha_0)\wedge\lambda_{\mathfrak{a}_1}(\beta_1)\}\\
&=\sup_{x=\alpha+\beta}\{r_{\mathfrak{a}_0}(\alpha_0)e^{i2\pi\omega_{\mathfrak{a}_0}(\alpha_0)}\wedge r_{\mathfrak{a}_1}(\beta_1)e^{i2\pi\omega_{\mathfrak{a}_1}(\beta_1)}\}\\
&=r_{\mathfrak{a}_0}(x_0)e^{i2\pi\omega_{\mathfrak{a}_0}(x_0)}\wedge r_{\mathfrak{a}_1}(x_1)e^{i2\pi\omega_{\mathfrak{a}_1}(x_1)}\\
&=r_{A_0}(x_0)e^{i2\pi\omega_{A_0}(x_0)}\wedge r_{A_1}(x_1)e^{i2\pi\omega_{A_1}(x_1)}\\
&=\lambda_{A_0}(x_0)\wedge\lambda_{A_1}(x_1).
\end{align*}
\textnormal{
(2) $A_0=(\lambda_{A_0}, \rho_{A_0})$ and $A_1=(\lambda_{A_1}, \rho_{A_1})$ are the even and odd parts of $A=(\lambda_A,\rho_A)$ (respectively)\\
(3) $\mathfrak{a}_0=(\lambda_{\mathfrak{a}_0}, \rho_{\mathfrak{a}_0})$ and $\mathfrak{a}_1=(\lambda_{\mathfrak{a}_1}, \rho_{\mathfrak{a}_1})$ are the extensions of $A_0=(\lambda_{A_0},\rho_{A_0})$ and $A_1=(\lambda_{A_1},\rho_{A_1})$.} (respectively)
\end{rmk}
\begin{defn}
Let $A=(\lambda_A,\rho_A)$  be CIF set of $V$. Then $A=(\lambda_A,\rho_A)$ is called a complex fuzzy lie sub-superalgebra of $V$, if it satisfies the following conditions:\\
(1) $A=(\lambda_A,\rho_A)$ is a $\mathbb{Z}_2$-graded CIF vector space\\
(2) $\lambda_A([x,y])\geq\lambda_A(x)\wedge \lambda_A(y)$ and
$\rho_A([x,y])\leq\rho_A(x)\vee\rho_A(y)$. If the condition(2) is replaced by (3) $\lambda_A([x,y])\geq\lambda_A(x)\vee \lambda_A(y)$ and $\rho_A([x,y])\leq\rho_A(x)\wedge\rho_A(y)$, then $A=(\lambda_A,\rho_A)$ is called a CIF ideal of $V$.
\end{defn}
\begin{Example}\textnormal{
Let $N=N_0\oplus N_1$, where $N_0=<e>$, $N_1=<a_1,...,a_n,b_1,...,b_n>$, and $[a_i,b_i]=e$, $i=1,...,n$, the remaining brackets being zero. Then, $N$ is Lie superalgebra by \cite[page 11]{V}.
Define $A_0=(\lambda_{A_0},\rho_{A_0})$, where $$\lambda_{A_0}:N_0\rightarrow\mathbb{C}\ \ by\ \  \lambda_{A_0}(x)=\begin{cases} 0.7e^{i(1.4)\pi} & :\quad 0\not=x\in N_0 \\ 1 & :\quad x=0\end{cases},$$ $$\rho_{A_0}:N_0\rightarrow\mathbb{C}\ \ by\ \ \rho_{A_0}(x)=\begin{cases} 0.2e^{i(0.4)\pi} & :\quad 0\not=x\in N_0 \\ 0 & :\quad x=0\end{cases}.$$ Also, define $A_1=(\lambda_{A_1},\rho_{A_1})$, where $$\lambda_{A_1}:N_1\rightarrow\mathbb{C}\ \ by\ \  \lambda_{A_1}(x)=\begin{cases} 0.5e^{i\pi} & :\quad 0\not=x\in N_1 \\ 1 & :\quad x=0\end{cases},$$ $$\rho_{A_1}:N_1\rightarrow\mathbb{C}\ \ by\ \ \rho_{A_1}(x)=\begin{cases} 0.4e^{i(0.8)\pi} & :\quad 0\not=x\in N_1 \\ 0 & :\quad x=0\end{cases}.$$ Define $A=(\lambda_A,\rho_A)$. Then, by Definition~\ref{def-2}, $A=A_0\oplus A_1$, and so $A=(\lambda_A,\rho_A)$ is a $\mathbb{Z}_2$- graded CIF vector subspace of $N$. Moreover, it is easy to check that $A=(\lambda_A,\rho_A)$ is a CIF ideal of $N$.}
\end{Example}
For any complex fuzzy set $\lambda =re^{i2\pi\omega}$ of $V$, we define the image of $\lambda$ by
$${\rm Im}(\lambda)=\{(t,s)\in[0,1]\times[0,1]\ :\ \lambda(x)=r(x)e^{i2\pi\omega(x)}=te^{i2\pi s}\ for\ some\ x\in V \}$$
\begin{defn}
For any $t,s\in [0, 1]$ and complex fuzzy subset $\lambda =re^{i2\pi\omega}$ of $V$, the set
\\
$U(\lambda,(t,s))=\left\{ x \in V\ |\ r(x)\geqslant t\ and\ \omega(x)\geqslant s\right\}$ is called an upper $(t, s)$-level cut of $\lambda$,
\\$L(\lambda,(t,s))=\left\{ x \in V\ |\ r(x)\leqslant t\ and\ \omega(x)\leqslant s\right\}$ is called a lower $(t, s)$-level cut of $\lambda$
\end{defn}
Suppose that $\lambda_A(x)=r_A(x)e^{i2\pi\omega_A(x)}$ and $\rho_A(x)=\hat{r}_A(x)e^{i2\pi\hat{\omega}_A(x)}$, then we have the following result
\begin{Theorem}\textnormal{
If $A=(\lambda_A, \rho_A)$ is a CIF lie sub-superalgebra (respectively CIF lie ideal) of $V$, then the sets $U(\lambda_A,(t, s))$ and $L(\rho_A,(t, s))$ are lie sub-superalgebras (respectively ideals) of $V$ for every $(t,s)\in{\rm Im}(\lambda_A) \cap{\rm Im}(\rho_A).$}
\end{Theorem}
\begin{proof}
Let $(t,s)\in{\rm Im}(\lambda_A)\cap{\rm Im}(\rho_A) \subseteq [0,1]\times[0,1]$ and let $x,y\in U(\lambda_A,(t,s))$ and let $\alpha\in k$. Because $A=(\lambda_A, \rho_A)$ is a CIF lie sub-superalgebra we have
$\lambda_A(x+y)=r_A(x+y)e^{i2\pi\omega(x+y)}\geqslant r_A(x)e^{i2\pi\omega_A(x)}\wedge r_A(y)e^{i2\pi\omega_A(y)}$

$=\lambda_A(x)\wedge\lambda_A(y)\geqslant te^{i2\pi s}$, since
$r_A(x+y)\geqslant r_A(x)\wedge r_A(y)\geqslant t$ and $\omega_A(x+y)\geqslant \omega_A(x)\wedge \omega_A(y)\geqslant s$ and
$\lambda_A(\alpha x)=r_A(\alpha x)e^{i2\pi\omega_A(\alpha x)}\geqslant r_A(x)e^{i2\pi\omega_A(x)}\geqslant te^{i2\pi s}$, since $r_A(\alpha x)\geqslant r_A(x)\geqslant t$ and $\omega_A(\alpha x)\geqslant \omega_A(x)\geqslant s$. Then  $x+y, \alpha x \in U(\lambda_A, (t,s))$. For any $x\in U(\lambda_A, (t, s))\subseteq V, x= x_0+x_1$, where $x_0 \in V_0, x_1\in V_1$. We show that $x_0, x_1 \in U(\lambda_A,(t,s))$.
since
\begin{eqnarray*}
\lambda_A(x)&=&r_A(x)e^{i2\pi\omega_A(x)}\\
&=&\lambda_{A_0}(x_0)\wedge \lambda_{A_1}(x_1)\\
&=&r_{A_0}(x_0)e^{i2\pi\omega_{A_0}(x_0)}\wedge r_{A_1}(x_1)e^{i2\pi\omega_{A_1}(x_1)}\\
&\geqslant& (r_{A_0}(x_0)\wedge r_{A_1}(x_1))e^{i2\pi(\omega_{A_0}(x_0)\wedge \omega_{A_1}(x_1))}\\
&\geqslant& te^{i2\pi s}.
\end{eqnarray*}
If $\lambda_{A_0}(x_0)\geqslant \lambda_{A_1}(x_1)$, we have
\begin{eqnarray*}
\lambda_A(x_1)&=&r_A(x_1)e^{i2\pi\omega_A(x_1)}\\
&=&r_{A_1}(x_1)e^{i2\pi\omega_{A_1}(x_1)}\\
&=&\lambda_{A_1}(x_1)\\
&\geqslant & te^{i2\pi s},so
\end{eqnarray*}
$x_1\in U(\lambda_A,(t, s))$ and\\
\begin{eqnarray*}
\lambda_A(x_0)&=&\lambda_{A_0}(x_0)\\
& = & r_{A_0}(x_0)e^{i2\pi\omega_{A_0}(x_0)}\\
&\geqslant&\lambda_{A_1}(x_1)\\
&\geqslant & te^{i2\pi s}, so
\end{eqnarray*}
$x_0\in U(\lambda_A,(t, s))$. Hence $U(\lambda_A,(t, s))$ is a $\mathbb{Z}_2$-graded vector subspace of $V$ for every $(t,s)\in{\rm Im}(\lambda_A)\cap{\rm Im}(\rho_A)$. Let $(t, s)\in{\rm Im}(\lambda_A)\cap {\rm Im}(\rho_A)$ and $x, y\in U(\lambda_A,(t, s))$. Then
$\lambda_A(x)\geqslant te^{i2\pi s}$, where $r_A(x)\geqslant t$ and $\omega_A(x)\geqslant s$ and $\lambda_A(y)\geqslant te^{i2\pi s}$, where $r_A(y)\geqslant t$ and $\omega_A(y)\geqslant s$. So $\lambda_A([x, y])\geqslant \lambda_A(x)\wedge \lambda_A(y)\geqslant te^{i2\pi s}$, which implies that $[x, y]\in U(\lambda_A, (t, s))$. Hence $(\lambda_A, (t, s))$ is a lie sub-superalgebra of $V$ for every $(t, s)\in{\rm Im}(\lambda_A) \cap{\rm Im}(\rho_A)$.\\
Also, for any $x\in L(\rho_A,(t, s))\subseteq V$, $x$ can be expressed as $x=x_0+x_1$, where $x_0\in V_0, x_1\in V_1$. We show that $x_0, x_1\in L(\rho_A,(t, s))$. since
\begin{eqnarray*}
\rho_A(x) &=&\hat{r}_A(x)e^{i2\pi\hat{\omega}_A(x)}\\
&=&\rho_{A_0}(x_0)\vee \rho_{A_1}(x_1)\\
&=& \hat{r}_{A_0}(x_0)e^{i2\pi \hat{\omega}_{A_0}(x_0)}\vee \hat{r}_{A_1}(x_1)e^{i2\pi \hat{\omega}_{A_1}(x_1)}\\
& \leqslant & (\hat{r}_{A_0}(x_0)\vee \hat{r}_{A_1}(x_1))e^{i2\pi (\hat{\omega}_{A_0}(x_0)\vee\hat{\omega}_{A_1}(x_1))}\\
&\leqslant & te^{i2\pi s},
\end{eqnarray*}
if $\rho_{A_0}(x_0)\geqslant \rho_{A_1}(x_1)$, we have $\rho_A(x_0)=\rho_{A_0}(x_0)\leqslant te^{i2\pi s}$, so $x_0\in L(\rho_A,(t, s))$ and $\rho_A(x_1)=\rho_{A_1}(x_1)\leqslant \rho_{A_0}(x_0)\leqslant te^{i2\pi s}$, so $x_1\in L(\rho_A, (t, s))$. Hence $L(\rho_A, (t, s))$ is a $\mathbb{Z}_2$-graded vector subspace of $V$ for any $(t, s)\in{\rm Im}(\lambda_A)\cap{\rm Im}(\rho_A)$.
Let $(t, s)\in{\rm Im}(\lambda_A)\cap{\rm Im}(\rho_A)$ and $x, y\in L(\rho_A,(t, s))$. Then $\rho_A(x)\leqslant te^{i2\pi s}$ and $\rho_A(y)\leqslant te^{i2\pi s}$. So $\rho_A([x, y])\leqslant \rho_A(x)\vee \rho_A(y) \leqslant te^{i2\pi s}$, which implies that $[x, y]\in L(\rho_A,(t, s))$. Hence $L(\rho_A,(t, s))$ is a lie sub-superalgebra of $V$ for any $(t, s)\in{\rm Im}(\lambda_A)\cap {\rm Im}(\rho_A)$.
\end{proof}
Next we suppose $A=(\lambda_A, \rho_A)$ is a CIF set of $V$, where $\lambda_A(x)=r_A(x)e^{i2\pi\omega_A(x)}$ and $\rho_A(x)=\hat{r}_A(x)e^{i2\pi\hat{\omega}_A(x)}$. Then define $\lambda_A^c(x)$ by $\lambda_A^c(x)=(1-r_A(x))e^{i2\pi(1-\omega_A(x))}$ and $\rho_A^c$ by $\rho_A^c=(1-\hat{r}_A(x))e^{i2\pi(1-\hat{\omega}_A(x))}$.
\begin{defn}$\ $\\
(1) $A^c=\{(x,\lambda_A(x), \lambda_A^c(x)): x\in V\}$. Shortly $A^c=(\lambda_A, \lambda_A^c)$.\\
(2) $A^L=\{(x,\rho_A^c(x) , \rho_A(x)): x\in V\}$. Shortly $A^L=(\rho_A^c,\rho_A)$.
\end{defn}
\begin{Theorem}\textnormal{\\
(1) If $A=(\lambda_A, \rho_A)$ is a CIF lie sub-suberalgebra (respectively CIF ideal) of $V$, then so is $A^c$.\\
(2) If $A=(\lambda_A, \rho_A)$ is a CIF lie sub-superalgebra (respectively CIF ideal) of $V$, then so is $A^L.$}
\end{Theorem}
\begin{proof}
(1) Since $(\lambda_A, \rho_A)$ is a $\mathbb{Z}_2$-graded CIF vector subspace of $V$, we have $A=A_0+A_1$, where $A_0=(\lambda_{A_0},\rho_{A_0}), A_1=(\lambda_{A_1}, \rho_{A_1})$ are CIF vector subspaces of $V_0$ and $V_1$ respectively, and for $x\in V$, $\lambda_A(x)=\sup_{x=a+b}\lbrace\lambda_{A_0}(a)\wedge \lambda_{A_1}(b)\rbrace.$
Define $A_0^c= (\lambda_{A_0}, \lambda_{A_0}^c)$, $A_1^c=(\lambda_{A_1}, \lambda_{A_1}^c)$ and
define $\mathfrak{a}_0^c= (\lambda_{\mathfrak{a}_0}, \lambda_{\mathfrak{a}_0}^c)$ and $\mathfrak{a}_1^c=(\lambda_{\mathfrak{a}_1}, \lambda_{\mathfrak{a}_1}^c)$.
Obviously, $(\mathfrak{a}_0^c, \mathfrak{a}_1^c)$ are the extensions of $(A_0^c, A_1^c)$. In order to prove
$A^c=\mathfrak{a}_0^c+\mathfrak{a}_1^c$ we only need to show that
$\lambda_A^c(x)=\inf_{x=a+b}\lbrace\lambda_{\mathfrak{a}_0}^c(a)\vee\lambda_{\mathfrak{a}_1}^c(b)\rbrace.$ Since $A$ is homogeneous indeed,
\begin{eqnarray*}
1-\lambda_A^c(x)&=&\sup_{x=a+b}\lbrace(1-\lambda_{\mathfrak{a}_0}^c(a))\wedge (1-\lambda_{\mathfrak{a}_1}^c(b))\rbrace\\
&=&\sup_{x=a+b}\lbrace(1-(\lambda_{\mathfrak{a}_0}^c(a)\vee\lambda_{\mathfrak{a}_1}^c(b))\rbrace\\
&=&\sup_{x=a+b}\lbrace1-((1-r_{\mathfrak{a}_0}(a))e^{i2\pi(1-w_{\mathfrak{a}_0}(a))}\vee(1-r_{\mathfrak{a}_1}(b))e^{i2\pi(1-\omega_{\mathfrak{a}_1}(b))})\rbrace\\
&=&\sup_{x=a+b}\lbrace1-((1-r_{\mathfrak{a}_0}(a)\vee1-r_{\mathfrak{a}_1}(b))e^{i2\pi((1-\omega_{\mathfrak{a}_0}(a))\vee(1-\omega_{\mathfrak{a}_1}(b)))})\rbrace\\
&=&1-\inf_{x=a+b}\lbrace1-r_{\mathfrak{a}_0}(a)\vee1-r_{\mathfrak{a}_1}(b)\rbrace e^{i2\pi(1-(1-\omega_{\mathfrak{a}_0}(a)\vee 1-w_{\mathfrak{a}_1}(b)))}\\
&=&1-\inf_{x=a+b}\lbrace (1-r_{\mathfrak{a}_0}(a))e^{i2\pi(1-w_{\mathfrak{a}_0}(a))}\vee(1-r_{\mathfrak{a}_1}(b))e^{i2\pi(1-\omega_{\mathfrak{a}_1}(b))}
\rbrace\\
&=&1-\inf_{x=a+b}\lbrace\lambda_{\mathfrak{a}_0}^c(a)\vee\lambda_{\mathfrak{a}_1}^c(b)\rbrace, so
\end{eqnarray*}
\[\lambda_A^c(x)=\underset{x=a+b}{\inf}\lbrace\lambda_{\mathfrak{a}_0}^c(a)\vee\lambda_{\mathfrak{a}_1}^c(b)\rbrace
=\lambda_{A_0}^c(x_0)\vee\lambda_{A_1}^c(x_1).\ \ \ \ \ \ \ \ \ \ \ \ \ \ \ \ \ \ \ \ \ \ \ \ \]
Moreover, it is easy to get $\lambda_{A'_0}^c(x)\vee\lambda_{A'_1}^c(x)=\begin{cases} 0 & :\quad x=0\\
1 & :\quad x\neq 0\end{cases}$,
we have that $A^c=A_0^c \oplus A_1^c$ is a $\mathbb{Z}_2$-graded CIF vector subspace of $V$.
Let $x,y\in V$. Since $\lambda_A([x,y])\geq\lambda_A(x)\wedge\lambda_A(y)$, we have $$1-\lambda_A([x,y])\geq(1-\lambda_A^c(x))\wedge(1-\lambda_A^c(y)),$$ and so $1-\lambda_A([x,y])\geq1-(\lambda_A^c(x)\vee\lambda_A^c(y))$. Hence $\lambda_A([x,y])\leq\lambda_A^c(x)\vee\lambda_A^c(y)$. So $A^c=(\lambda_A, \lambda_A^c)$ is a CIF lie sub-superalgebra of $V$. By the similar way we can prove the case of CIF ideal.\\
(2) The proof is similar to the proof of (1).
\end{proof}
By using the above result it is not difficult to verify that the following theorem is valid.
\begin{Theorem}\textnormal{
$A=(\lambda_A, \rho_A)$ is a CIF lie sub-superalgebra (respectively CIF ideal) of V If and only If $A^c$ and $A^{L}$ are CIF lie sub-superalgebras (respectively CIF ideals) of $V$.}
\end{Theorem}
\begin{Theorem}\textnormal{
If $A=(\lambda_A,\rho_A)$ is a CIF set of $V$ such that all non-empty level sets $U(\lambda_A,(t, s))$ and $L(\rho_A,(t, s))$ are lie sub-superalgebras (respectively ideals)of $V$, then $A=(\lambda_A, \rho_A)$ is CIF lie sub-superalgebra (respectively CIF ideal) of $V$.}
\end{Theorem}
\begin{proof}
Let $x, y\in V$ and $K$. We may assume that $\lambda_A(y)\geq \lambda_A(x)= t_1e^{i2\pi s_1}\and \ \rho_A(y)\leq \rho_A(x)=t_0e^{i2\pi s_0},$ then $x, y\in U(\lambda_A,(t_1, s_1))$ and $x, y\in L(\rho_A,(t_0, s_0))$. Since $U(\lambda_A, (t_1, s_1))$ and $L(\rho_A,(t_0, s_0))$ are vector subspaces of $V$, we get $x+y, \alpha x\in U(\lambda_A, (t_1, s_1))$ and $x+y, \alpha x \in L(\rho_A,(t_0, s_0))$. So, $\lambda_A(\alpha x)\geq \lambda_A(x)=t_1e^{i2\pi s_1}$ and $\lambda_A(x+y)\geq t_1e^{i2\pi s_1}=\lambda_A(x)\wedge \lambda_A(y)$, $\rho_A(\alpha x)\leqslant t_0e^{i2\pi s_0}=\rho_A(x)$ and $\rho_A(x+y)\leqslant t_0e^{i2\pi s_0}=\rho_A(x)\vee \rho_A(y)$.
Now, we show that $A=(\lambda_A, \rho_A)$ has a $\mathbb{Z}_2$-graded structure.
Define $A_0=(\lambda_{A_0}, \rho_{A_0})$ where $\lambda_{A_0}: V_0 \rightarrow \mathbb{C}$ by $x \mapsto \lambda_A(x)$, $ \rho_{A_0}: V_0\rightarrow \mathbb{C}$, by $x \mapsto \rho_A(x)$ and define $A_1=(\lambda_{A_1}, \rho_{A_1})$ where $\lambda_{A_1}: V_1 \rightarrow \mathbb{C}$ by $x\mapsto \lambda_A(x)$, $\rho_{A_1}: V_1 \rightarrow \mathbb{C}$ by $x\mapsto \rho_A(x)$. We extend $A_0=(\lambda_{A_0}, \rho_{A_0})$, $A_1=(\lambda_{A_1}, \rho_{A_1})$ to $\mathfrak{a}_0=(\lambda_{\mathfrak{a}_0}, \rho_{\mathfrak{a}_0})$, $\mathfrak{a}_1=(\lambda_{\mathfrak{a}_1}, \rho_{\mathfrak{a}_1})$ as follows. Define $\mathfrak{a}_0=(\lambda_{\mathfrak{a}_0}, \rho_{\mathfrak{a}_0})$ by
$$\lambda_{\mathfrak{a}_0}(x)=\begin{cases}\lambda_{A_0}(x) & if \quad x\in V_0\\ 0 & if \quad  x\not\in V_0
\end{cases} \ , \ \rho_{\mathfrak{a}_0}(x)=\begin{cases}\rho_{A_0}(x) & if \quad x\in V_0\\ 1 & if \quad  x\not\in V_0 \end{cases}. $$
And we define $\mathfrak{a}_1=(\lambda_{\mathfrak{a}_1}, \rho_{\mathfrak{a}_1})$ by
$$\lambda_{\mathfrak{a}_1}(x)=\begin{cases}\lambda_{A_1}(x) & if \quad x\in V_1\\ 0 & if \quad  x\not\in V_1
\end{cases} \ , \ \rho_{\mathfrak{a}_1}(x)=\begin{cases}\rho_{A_1}(x) & if \quad x\in V_1\\ 1 & if \quad  x\not\in V_1 \end{cases} $$
Then it is obvious that $\mathfrak{a}_0$, $\mathfrak{a}_1$ are CIF vector subspaces of $V$ and for any $0\neq x\in V$ we have $\mathfrak{a}_0 \cap \mathfrak{a}_1=(\lambda_{\mathfrak{a}_0}(x)\wedge  \lambda_{\mathfrak{a}_1}(x), \rho_{\mathfrak{a}_0}(x)\vee \rho_{\mathfrak{a}_1}(x))=(0,1).$
To show that $A=A_0\oplus A_1$, let $x\in V$. We may assume that $\lambda_A(x)=te^{i2\pi s}$, then $x \in U(\lambda_A, (t, s))$. Because $U(\lambda_A,(t, s))$ is a $\mathbb{Z}_2$ -graded vector subspace of $V$, we have $x=x_0+x_1$, where $x_0\in V_0 \cap U(\lambda_A, (t, s))$ and $x_1\in V_1\cap U(\lambda_A,(t, s))$. Since $te^{i2\pi s}=\lambda_A(x)=\lambda_A(x_0+x_1)\geq \lambda_A(x_0)\wedge \lambda_A(x_1)$, if $ \lambda_A(x_0)\geq \lambda_A(x_1)$, then $te^{i2\pi s}\geq \lambda_A(x_1)\geq te^{i2\pi s}$, we have $\lambda_A(x_1)=te^{i2\pi s}.$ Similarly, if $\lambda_A(x_1)\geq \lambda_A(x_0)$, we have $\lambda_A(x_0)=te^{i2\pi s}$. Hence
\begin{eqnarray*}
\lambda_A(x) & = & te^{i2\pi s}\\
& = &\lbrace \lambda_A(x_0)\wedge \lambda_A(x_1)| x=x_0+x_1\rbrace \\
& = & \lbrace\lambda_{A_0}(x_0)\wedge \lambda_{A_1}(x_1)| x= x_0+x_1\rbrace\\
 & = & \sup_{x=a+b}\lbrace \lambda_{\mathfrak{a}_0}(a)\wedge \lambda_{\mathfrak{a}_1}(b)\rbrace \\
 & = & \lambda_{\mathfrak{a}_0+\mathfrak{a}_1}(x)\\
 & = & \lambda_{A_0\oplus A_1}(x)\ \ \ \ \ \ \ \ ( \ Since\  \mathfrak{a}_0\cap\mathfrak{a}_1=(0, 1)\ ) \
\end{eqnarray*}
Also, we may assume that $\rho_A(x)=te^{i2\pi s}$, then $x\in L(\rho_A,(t, s)).$ Because $L(\rho_A,(t, s))$ is a $\mathbb{Z}_2$-graded vector subspace of $V$, we have $x=x_0+x_1$, where $x_0\in V_0\cap L(\rho_A, (t, s))$ and $x_1\in V_1\cap L(\rho_A,(t, s))$. Since $te^{i2\pi s}=\rho_A(x)=\rho_A(x_0+x_1)\leq \rho_A(x_0)\wedge\rho_A(x_1)$, if $\rho_A(x_0)\geq \rho_A(x_1)$, then $te^{i2\pi s}\leq \rho_A(x_0)\leq te^{i2\pi s}$, we have $\rho_A(x_0)=te^{i2\pi s}$. Similarly, if $\rho_A(x_1)\geq \rho_A(x_0)$, we have $\rho_A(x_1)=te^{i2\pi s}$. Hence
\begin{eqnarray*}
\rho_A(x) & = & te^{i2\pi s}\\
&=&\lbrace \rho_A(x_0)\vee \rho_A(x_1)| x=x_0+x_1\rbrace\\
&=&\lbrace \rho_{A_0}(x_0)\vee \rho_{A_1}(x_1)|x=x_0+x_1\rbrace\\
&=&\inf_{x=a+b}\lbrace \rho_{\mathfrak{a}_0}(a)\vee \rho_{\mathfrak{a}_1}(b)\rbrace\\
&=& \rho_{\mathfrak{a}_0+\mathfrak{a}_1}(x)\\
&=&\rho_{A_0\oplus A_1}(x)
\end{eqnarray*}
So $A=A_0\oplus A_1$, and hence $A=(\lambda_A, \rho_A)$ is a CIF vector subspace of $V$.
Let $x, y\in V$ and assume that $\lambda_A(y)\geq \lambda_A(x)\geq te^{i2\pi s}$, where $t,s\in [0, 1]$. Then $x, y\in U(\lambda_A,(t, s))$. Because $U(\lambda_A,(t, s))$ is a lie sub-superalgebra of $V$, we get $[x, y]\in (\lambda_A,(t, s))$, then $\lambda_A([x, y])\geq te^{i2\pi s}=\lambda_A(x)\wedge \lambda_A(y)$.\\
We also assume that $\rho_A(x)\leq te^{i2\pi s}\leq \rho_A(y)$ for some $t,s\in [0, 1]$, then $x, y \in L(\rho_A,(t, s))$. Because $L(\rho_A,(t, s))$ is a lie sub-superalgebra of $V$. We get $[x, y]\in L(\rho_A,(t, s))$, then $\rho_A([x, y])\leq te^{i2\pi s}\leq \rho_A(x)\vee \rho_A(y)$. Thus $A=(\lambda_A, \rho_A)$ is CIF lie sub-superalgebra of $V$. The case of CIF ideal is similar to show.
\end{proof}
Let $A=(\lambda_A, \rho_A)$ and $B=(\lambda_B, \rho_B)$ be $CIF$ vector subspaces of $V$, where $\lambda_A=r_Ae^{i2\pi\omega_A},\lambda_B=r_Be^{i2\pi\omega_B}$and $\rho_A=\hat{r}_Ae^{i2\pi\hat{\omega}_A}, \rho_B=\hat{r}_Be^{i2\pi\hat{\omega}_B}$. We have to remark that if $A$ is homogenous with $B$ then, by definition~\ref{def-1}, the $CIF$ set $A+B=(\lambda_{A+B}, \rho_{A+B})$ of $V$ is defined by
\begin{align*}
\lambda_{A+B}(x)&=\sup_{x=a+b}\{(r_A(a)\wedge r_B(b))\}e^{i2\pi\sup_{x=a+b}\lbrace (\omega_A(a)\wedge\omega_B(b))\rbrace}\\
&=r_{A+B}(x)e^{i2\pi\omega_{A+B}(x)},
\end{align*}
and
\begin{align*}
\rho_{A+B}(x)&=\inf_{x=a+b}\{(\hat{r}_A(a)\vee\hat{r}_B(b))\}e^{i2\pi\inf_{x=a+b}\lbrace (\hat{\omega}_A(a)\vee\hat{\omega}_B(b))\rbrace}\\
&=\hat{r}_{A+B}(x)e^{i2\pi\hat{\omega}_{A+B}(x)}.
\end{align*}
Now we have the following two results:
\begin{Theorem}
If $A=(\lambda_A, \rho_A)$ and $B=(\lambda_B, \rho_B)$ are CIF lie sub-superalgebras (respectively CIF ideals) of $V=V_0+V_1$, then so is $A+B=(\lambda_{A+B}, \rho_{A+B})$.
\end{Theorem}
\begin{proof}
For $\alpha=0,1$, define $(A+B)_\alpha=(\lambda_{(A+B)_\alpha},\rho_{(A+B)_\alpha})$, where $\lambda_{(A+B)_\alpha}=\lambda_{A_\alpha}+\lambda_{B_\alpha}$ and $\rho_{(A+B)_\alpha}=\rho_{A_\alpha}+\rho_{B_\alpha}$. By Lemma~\ref{mylemma-1} we know that they are CIF subspaces of $V_\alpha$ ($\alpha=0,1$).\\
Again for $\alpha=0,1$, define $(\mathfrak{a}+\mathfrak{b})_\alpha=(\lambda_{(\mathfrak{a}+\mathfrak{b})_\alpha},\rho_{(\mathfrak{a}+\mathfrak{b})_\alpha})$, where $\lambda_{(\mathfrak{a}+\mathfrak{b})_\alpha}=\lambda_{\mathfrak{a}_\alpha}+\lambda_{\mathfrak{b}_\alpha}$ and $\rho_{(\mathfrak{a}+\mathfrak{b})_\alpha}=\rho_{\mathfrak{a}_\alpha}+\rho_{\mathfrak{b}_\alpha}$. Obviously, $(\mathfrak{a}+\mathfrak{b})_\alpha$ are extensions of $(A+B)_\alpha$ for $\alpha=0,1$ (respectively). Let $x\in V$. Then
\begin{align*}
&\lambda_{(A+B)}(x)=\sup_{x=a+b}\{\lambda_A(a)\wedge\lambda_B(b)\}\\
&=\sup_{x=a+b}\{\lambda_{\mathfrak{a}_0+\mathfrak{a}_1}(a)\wedge\lambda_{\mathfrak{b}_0+\mathfrak{b}_1}(b)\}\\
&=\sup_{x=a+b}\{\sup_{a=m+n}\{\lambda_{\mathfrak{a}_0}(m)\wedge\lambda_{\mathfrak{a}_1}(n)\}\wedge\sup_{b=k+l}\{\lambda_{\mathfrak{b}_0}(k)\wedge\lambda_{\mathfrak{b}_1}(l)\}\}\\
&=\sup_{x=a+b}\{\sup_{a=m+n}\{r_{\mathfrak{a}_0}(m)e^{i2\pi\omega_{\mathfrak{a}_0}(m)}\wedge r_{\mathfrak{a}_1}(n)e^{i2\pi\omega_{\mathfrak{a}_1}(n)}\}
\wedge\sup_{b=k+l}\{r_{\mathfrak{b}_0}(k)e^{i2\pi\omega_{\mathfrak{b}_0}(k)}\wedge r_{\mathfrak{b}_1}(l)e^{i2\pi\omega_{\mathfrak{b}_1}(l)}\}\}\\
&=\sup_{x=a+b}\{\sup_{a+b=m+n+k+l}\{r_{\mathfrak{a}_0}(m)e^{i2\pi\omega_{\mathfrak{a}_0}(m)}\wedge r_{\mathfrak{a}_1}(n)e^{i2\pi\omega_{\mathfrak{a}_1}(n)}
\wedge r_{\mathfrak{b}_0}(k)e^{i2\pi\omega_{\mathfrak{b}_0}(k)}\wedge r_{\mathfrak{b}_1}(l)e^{i2\pi\omega_{\mathfrak{b}_1}(l)}\}\}\\
&=\sup_{a+b=m+n+k+l}\{r_{\mathfrak{a}_0}(m)e^{i2\pi\omega_{\mathfrak{a}_0}(m)}\wedge r_{\mathfrak{b}_0}(k)e^{i2\pi\omega_{\mathfrak{b}_0}(k)}\}
\wedge\sup_{a+b=m+n+k+l}\{r_{\mathfrak{a}_1}(n)e^{i2\pi\omega_{\mathfrak{a}_1}(n)}\wedge r_{\mathfrak{b}_1}(l)e^{i2\pi\omega_{\mathfrak{b}_1}(l)}\}\\
&=\sup_{a+b=m+n+k+l}\{\sup_{m+k}\{(r_{\mathfrak{a}_0}(m)\wedge r_{\mathfrak{b}_0}(k))e^{i2\pi(\omega_{\mathfrak{a}_0}(m)\wedge \omega_{\mathfrak{b}_0}(k))}\}
\wedge\sup_{n+l}\{(r_{\mathfrak{a}_1}(n)\wedge r_{\mathfrak{b}_1}(l))e^{i2\pi(\omega_{\mathfrak{a}_1}(n)\wedge \omega_{\mathfrak{b}_1}(l))}\}\}\\
&=\sup_{x=m+n+k+l}\{r_{\mathfrak{a}_0+\mathfrak{b}_0}(m+k)e^{i2\pi\omega_{\mathfrak{a}_0+\mathfrak{b}_0}(m+k)}\wedge r_{\mathfrak{a}_1+\mathfrak{b}_1}(n+l)e^{i2\pi\omega_{\mathfrak{a}_1+\mathfrak{b}_1}(n+l)}\}\\
&=r_{\mathfrak{a}_0+\mathfrak{b}_0+\mathfrak{a}_1+\mathfrak{b}_1}(x)e^{i2\pi\omega_{\mathfrak{a}_0+\mathfrak{b}_0+\mathfrak{a}_1+\mathfrak{b}_1}(x)}\\
&=\lambda_{(\mathfrak{a}+\mathfrak{b})_0+(\mathfrak{a}+\mathfrak{b})_1}(x).
\end{align*}
and
\begin{align*}
&\rho_{(A+B)}(x)=\inf_{x=a+b}\{\rho_A(a)\vee\rho_B(b)\}\\
&=\inf_{x=a+b}\{\rho_{\mathfrak{a}_0+\mathfrak{a}_1}(a)\vee\rho_{\mathfrak{b}_0+\mathfrak{b}_1}(b)\}\\
&=\inf_{x=a+b}\{\inf_{a=m+n}\{\rho_{\mathfrak{a}_0}(m)\vee\rho_{\mathfrak{a}_1}(n)\}\vee\inf_{b=k+l}\{\rho_{\mathfrak{b}_0}(k)\vee\rho_{\mathfrak{b}_1}(l)\}\}\\
&=\inf_{x=a+b}\{\inf_{a=m+n}\{\hat{r}_{\mathfrak{a}_0}(m)e^{i2\pi\hat{\omega}_{\mathfrak{a}_0}(m)}\vee \hat{r}_{\mathfrak{a}_1}(n)e^{i2\pi\hat{\omega}_{\mathfrak{a}_1}(n)}\}
\vee\inf_{b=k+l}\{\hat{r}_{\mathfrak{b}_0}(k)e^{i2\pi\hat{\omega}_{\mathfrak{b}_0}(k)}\vee \hat{r}_{\mathfrak{b}_1}(l)e^{i2\pi\hat{\omega}_{\mathfrak{b}_1}(l)}\}\}\\
&=\inf_{x=a+b}\{\inf_{a+b=m+n+k+l}\{\hat{r}_{\mathfrak{a}_0}(m)e^{i2\pi\hat{\omega}_{\mathfrak{a}_0}(m)}\vee \hat{r}_{\mathfrak{a}_1}(n)e^{i2\pi\hat{\omega}_{\mathfrak{a}_1}(n)}
\vee\hat{r}_{\mathfrak{b}_0}(k)e^{i2\pi\hat{\omega}_{\mathfrak{b}_0}(k)}\vee \hat{r}_{\mathfrak{b}_1}(l)e^{i2\pi\hat{\omega}_{\mathfrak{b}_1}(l)}\}\}\\
&=\inf_{a+b=m+n+k+l}\{\hat{r}_{\mathfrak{a}_0}(m)e^{i2\pi\hat{\omega}_{\mathfrak{a}_0}(m)}\vee \hat{r}_{\mathfrak{b}_0}(k)e^{i2\pi\hat{\omega}_{\mathfrak{b}_0}(k)}\}
\vee\inf_{a+b=m+n+k+l}\{\hat{r}_{\mathfrak{a}_1}(n)e^{i2\pi\hat{\omega}_{\mathfrak{a}_1}(n)}\vee \hat{r}_{\mathfrak{b}_1}(l)e^{i2\pi\hat{\omega}_{\mathfrak{b}_1}(l)}\}\\
&=\inf_{a+b=m+n+k+l}\{\inf_{m+k}\{(\hat{r}_{\mathfrak{a}_0}(m)\vee \hat{r}_{\mathfrak{b}_0}(k))e^{i2\pi(\hat{\omega}_{\mathfrak{a}_0}(m)\vee\hat{\omega}_{\mathfrak{b}_0}(k))}\}
\vee\inf_{n+l}\{(\hat{r}_{\mathfrak{a}_1}(n)\vee \hat{r}_{\mathfrak{b}_1}(l))e^{i2\pi(\hat{\omega}_{\mathfrak{a}_1}(n)\vee \hat{\omega}_{\mathfrak{b}_1}(l))}\}\}\\
&=\inf_{x=m+n+k+l}\{\hat{r}_{\mathfrak{a}_0+\mathfrak{b}_0}(m+k)e^{i2\pi\hat{\omega}_{\mathfrak{a}_0+\mathfrak{b}_0}(m+k)}\vee \hat{r}_{\mathfrak{a}_1+\mathfrak{b}_1}(n+l)e^{i2\pi\hat{\omega}_{\mathfrak{a}_1+\mathfrak{b}_1}(n+l)}\}\\
&=\hat{r}_{\mathfrak{a}_0+\mathfrak{b}_0+\mathfrak{a}_1+\mathfrak{b}_1}(x)e^{i2\pi\hat{\omega}_{\mathfrak{a}_0+\mathfrak{b}_0+\mathfrak{a}_1+\mathfrak{b}_1}(x)}\\
&=\rho_{(\mathfrak{a}+\mathfrak{b})_0+(\mathfrak{a}+\mathfrak{b})_1}(x).
\end{align*}
Moreover if $0\not=x\in V$ then
\begin{eqnarray*}
\lambda_{(\mathfrak{a}+\mathfrak{b})_0}(x)\wedge\lambda_{(\mathfrak{a}+\mathfrak{b})_1}(x)&=&\sup_{x=a+b}\{\lambda_{\mathfrak{a}_0}(a)\wedge\lambda_{\mathfrak{b}_0}(b)\}\wedge\sup_{x=a+b}\{\lambda_{\mathfrak{a}_1}(a)\wedge\lambda_{\mathfrak{b}_1}(b)\}\\
&=&0
\end{eqnarray*}
\begin{eqnarray*}
\rho_{(\mathfrak{a}+\mathfrak{b})_0}(x)\vee\rho_{(\mathfrak{a}+\mathfrak{b})_1}(x)&=&\inf_{x=a+b}\{\rho_{\mathfrak{a}_0}(a)\vee\rho_{\mathfrak{b}_0}(b)\}\vee\inf_{x=a+b}\{\rho_{\mathfrak{a}_1}(a)\vee\rho_{\mathfrak{b}_1}(b)\}\\
&=&1
\end{eqnarray*}
So $A+B$ is a $\mathbb{Z}_2$-CIF vector subspaces of $V$.\\
(1) Let $x,y\in V$ we need to show that $\lambda_{A+B}([x,y])\geq\lambda_{A+B}(x)\vee\lambda_{A+B}(y)$ and
$\rho_{A+B}([x,y])\leq\rho_{A+B}(x)\wedge\rho_{A+B}(y)$. Suppose that $\lambda_{A+B}([x,y])<\lambda_{A+B}(x)\vee\lambda_{A+B}(y)$, so without loss of generality we may assume that $\lambda_{A+B}([x,y])<\lambda_{A+B}(x)$. Then  $\lambda_{A+B}([x,y])=r_{A+B}([x,y])e^{i2\pi\omega_{A+B}([x,y])}<r_{A+B}(x)e^{i2\pi\omega_{A+B}(x)}$, and because $A+B$ is homogenous, then we have that $r_{A+B}([x,y])<r_{A+B}(x)$ or $\omega_{A+B}([x,y])<\omega_{A+B}(x)$. Again without loss of generality we may assume that $r_{A+B}([x,y])<r_{A+B}(x)$. Choose a number $t\in[0,1]$, such that $r_{A+B}([x,y])<t<r_{A+B}(x)$. Then there exist $a,b\in V$ with $x=a+b$, such that  $r_A(a)>t$ and $r_B(b)>t$. So
\begin{eqnarray*}
r_{A+B}([x,y])&=&\sup_{[x,y]=[a^{'},y]+[b^{'},y]}\{r_A([a^{'},y])\wedge r_B([b^{'},y])\}\\
&\geq&\sup_{[x,y]=[a^{'},y]+[b^{'},y]}\{r_A(a^{'})\wedge r_B(b^{'})\}\ (A,B\ are\ ideals)\\
&>&t>r_{A+B}([x,y]).
\end{eqnarray*}
Which is a contradiction. Also suppose that $\rho_{A+B}([x,y])>\rho_{A+B}(x)\wedge\rho_{A+B}(y)$. Then $\rho_{A+B}([x,y])>\rho_{A+B}(x)$ or $\rho_{A+B}([x,y])>\rho_{A+B}(y)$ and without loss of generality we may assume that $\rho_{A+B}([x,y])=\hat{r}_{A+B}([x,y])e^{i2\pi\hat{\omega}_{A+B}([x,y])}>\rho_{A+B}(x)=\hat{r}_{A+B}(x)e^{i2\pi\hat{\omega}_{A+B}(x)}$ because $A+B$ is homogenous, we may assume that $\hat{r}_{A+B}([x,y])>\hat{r}_{A+B}(x)$. Choose a number $t\in[0,1]$, such that $\hat{r}_{A+B}([x,y])>t>\hat{r}_{A+B}(x)$. Then there exist $a,b\in V$ with $x=a+b$, such that  $\hat{r}_A(a)<t$ and $\hat{r}_B(b)<t$. So
\begin{eqnarray*}
\hat{r}_{A+B}([x,y])&=&\inf_{[x,y]=[a^{'},y]+[b^{'},y]}\{\hat{r}_A([a^{'},y])\vee \hat{r}_B([b^{'},y])\}\\
&\leq&\inf_{[x,y]=[a^{'},y]+[b^{'},y]}\{\hat{r}_A(a^{'})\vee\hat{r}_B(b^{'})\}\ (A,B\ are\ ideals)\\
&<&t<\hat{r}_{A+B}([x,y]).
\end{eqnarray*}
Which is a contradiction. Therefore, $A+B=(\lambda_{A+B},\rho_{A+B})$ is a CIF ideal of $V$.\\
(2) Let $x,y\in V$ we need to show that $\lambda_{A+B}([x,y])\geq\lambda_{A+B}(x)\wedge\lambda_{A+B}(y)$ and
$\rho_{A+B}([x,y])\leq\rho_{A+B}(x)\vee\rho_{A+B}(y)$. Suppose that $\lambda_{A+B}([x,y])<\lambda_{A+B}(x)\wedge\lambda_{A+B}(y)$, then
$$r_{A+B}([x,y])e^{i2\pi\omega_{A+B}([x,y])}<r_{A+B}(x)e^{i2\pi\omega_{A+B}(x)}\wedge r_{A+B}(y)e^{i2\pi\omega_{A+B}(y)}$$ and so, $r_{A+B}([x,y])e^{i2\pi\omega_{A+B}([x,y])}<(r_{A+B}(x)\wedge r_{A+B}(y)) e^{i2\pi(\omega_{A+B}(x)\wedge\omega_{A+B}(y))}$, hence $$r_{A+B}([x,y])<r_{A+B}(x)\wedge r_{A+B}(y)\ \ \ {\rm or}\ \ \ \omega_{A+B}([x,y])<\omega_{A+B}(x)\wedge\omega_{A+B}(y).$$ If $r_{A+B}([x,y])<r_{A+B}(x)\wedge r_{A+B}(y)$, then $r_{A+B}([x,y])<r_{A+B}(x)$ and $r_{A+B}([x,y])<r_{A+B}(y)$. Choose a number $t\in[0,1]$, such that $r_{A+B}([x,y])<t<r_{A+B}(x)\wedge r_{A+B}(y)$. Then there exist $a,b,c,d\in V$ with $x=a+b$ and $y=c+d$, such that $r_A(a)>t$, $r_B(b)>t$, $r_A(c)>t$, $r_B(d)>t$. So
\begin{eqnarray*}
r_{A+B}([x,y])&=&\sup_{[x,y]=[a^{'},y]+[b^{'},y]}\{r_A([a^{'},y])\wedge r_B([b^{'},y])\}\\
&=&\sup_{[x,y]=[a^{'}+b^{'},c^{'}+d^{'}]}\{r_A([a^{'},c^{'}+d^{'}])\wedge r_B([b^{'},c^{'}+d^{'}])\}\\
&=&\sup_{[x,y]=[a^{'}+b^{'},c^{'}+d^{'}]}\{r_A([a^{'},c^{'}])\wedge r_A([a^{'},d^{'}])\wedge r_B([b^{'},c^{'}])\wedge r_B([b^{'},d^{'}])\}\\
&=&\sup_{x=a^{'}+b^{'},y=c^{'}+d^{'}}\{\sup_{[a^{'},c^{'}+d^{'}]}\{r_A(a^{'})\wedge r_A(c^{'})\wedge r_A(d^{'})\}\wedge\sup_{[a^{'},c^{'}+d^{'}]}\{r_A(a^{'})\wedge r_A(c^{'})\wedge r_A(d^{'})\}\}\\
&=&\sup_{x=a^{'}+b^{'}}\{r_A(a^{'})\wedge r_B(b^{'})\}\wedge\sup_{y=c^{'}+d^{'}}\{r_A(c^{'})\wedge r_A(d^{'})\}\wedge\sup_{y=c^{'}+d^{'}}\{r_B(c^{'})\wedge r_B(d^{'})\}\\
&\geq&\sup_{x=a^{'}+b^{'}}\{r_A(a^{'})\wedge r_B(b^{'})\}\wedge r_A(c)\wedge r_B(d)\\
&\geq&r_A(a)\wedge r_B(b)\wedge r_A(c)\wedge r_B(d)\\
&>&t>r_{A+B}([x,y]).
\end{eqnarray*}
Which is a contradiction. The other case can be proved similarly, so $\lambda_{A+B}([x,y])\geq\lambda_{A+B}(x)\wedge\lambda_{A+B}(y)$.  Also suppose that $\rho_{A+B}([x,y])>\rho_{A+B}(x)\vee\rho_{A+B}(y)$, then $$\hat{r}_{A+B}([x,y])e^{i2\pi\hat{\omega}_{A+B}([x,y])}>(\hat{r}_{A+B}(x)\vee\hat{r}_{A+B}(y)) e^{i2\pi(\hat{\omega}_{A+B}(x)\vee\hat{\omega}_{A+B}(y))},$$ hence $\hat{r}_{A+B}([x,y])>\hat{r}_{A+B}(x)\vee\hat{r}_{A+B}(y)$ or $\hat{\omega}_{A+B}([x,y])>\hat{\omega}_{A+B}(x)\vee\hat{\omega}_{A+B}(y)$. Since $A,B$ are homogenous without loss of generality we may assume that $\hat{r}_{A+B}([x,y])>\hat{r}_{A+B}(x)\vee\hat{r}_{A+B}(y)$. Choose a number $t\in[0,1]$, such that $\hat{r}_{A+B}([x,y])>t>\hat{r}_{A+B}(x)\vee\hat{r}_{A+B}(y)$. Then there exist $a,b,c,d\in V$ with $x=a+b$ and $y=c+d$, such that  $\hat{r}_A(a)<t$, $\hat{r}_B(b)<t$ and $\hat{r}_A(c)<t$, $\hat{r}_B(d)<t$. So
\begin{eqnarray*}
\hat{r}_{A+B}([x,y])&=&\inf_{[x,y]=[a^{'},y]+[b^{'},y]}\{\hat{r}_A([a^{'},y])\vee \hat{r}_B([b^{'},y])\}\\
&=&\inf_{[x,y]=[a^{'}+b^{'},c^{'}+d^{'}]}\{\hat{r}_A([a^{'},c^{'}+d^{'}])\vee \hat{r}_B([b^{'},c^{'}+d^{'}])\}\\
&=&\inf_{[x,y]=[a^{'}+b^{'},c^{'}+d^{'}]}\{\hat{r}_A([a^{'},c^{'}])\vee\hat{r}_A([a^{'},d^{'}])\vee \hat{r}_B([b^{'},c^{'}])\vee\hat{r}_B([b^{'},d^{'}])\}\\
&=&\inf_{x=a^{'}+b^{'},y=c^{'}+d^{'}}\{\inf_{[a^{'},c^{'}+d^{'}]}\{\hat{r}_A(a^{'})\vee \hat{r}_A(c^{'})\vee\hat{r}_A(d^{'})\}\vee\inf_{[a^{'},c^{'}+d^{'}]}\{\hat{r}_A(a^{'})\vee \hat{r}_A(c^{'})\vee\hat{r}_A(d^{'})\}\}\\
&=&\inf_{x=a^{'}+b^{'}}\{\hat{r}_A(a^{'})\vee\hat{r}_B(b^{'})\}\vee\inf_{y=c^{'}+d^{'}}\{\hat{r}_A(c^{'})\vee \hat{r}_A(d^{'})\}\vee\inf_{y=c^{'}+d^{'}}\{\hat{r}_B(c^{'})\vee\hat{r}_B(d^{'})\}\\
&\leq&\inf_{x=a^{'}+b^{'}}\{\hat{r}_A(a^{'})\vee\hat{r}_B(b^{'})\}\vee\hat{r}_A(c)\vee\hat{r}_B(d)\\
&\leq&\hat{r}_A(a)\vee\hat{r}_B(b)\vee\hat{r}_A(c)\vee\hat{r}_B(d)\\
&<&t<\hat{r}_{A+B}([x,y]).
\end{eqnarray*}
Which is a contradiction. The other case can be proved similarly, so $\rho_{A+B}([x,y])\leq\rho_{A+B}(x)\vee\rho_{A+B}(y)$. Therefore, $A+B=(\lambda_{A+B},\rho_{A+B})$ is a CIF lie subsuperalgebra of $V$.
\end{proof}
\begin{Theorem}
If $A=(\lambda_A, \rho_A)$ and $B=(\lambda_B, \rho_B)$ are $CIF$ lie sub-superalgebras (respectively CIF ideals) of $V=V_0+V_1$, then so is $A\cap B=(\lambda_{A\cap B}, \rho_{A\cap B})$
\end{Theorem}
\begin{proof}
Since $A=A_0\oplus A_1$ and $B=B_0\oplus B_1$, for $\alpha=0,1$ we can define $(A\cap B)_\alpha=(\lambda_{(A\cap B)_\alpha},\rho_{(A\cap B)_\alpha})$, where $\lambda_{(A\cap B)_\alpha}=\lambda_{A_\alpha}\cap\lambda_{B_\alpha}$ and $\rho_{(A\cap B)_\alpha}=\rho_{A_\alpha}\cap\rho_{B_\alpha}$. By Lemma~\ref{mylemma-2} we know that they are CIF subspaces of $V_\alpha$.\\
Again for $\alpha=0,1$, define $(\mathfrak{a}\cap\mathfrak{b})_\alpha=(\lambda_{(\mathfrak{a}\cap \mathfrak{b})_\alpha},\rho_{(\mathfrak{a}\cap\mathfrak{b})_\alpha})$, where $\lambda_{(\mathfrak{a}\cap\mathfrak{b})_\alpha}=\lambda_{\mathfrak{a}_\alpha}\cap\lambda_{\mathfrak{b}_\alpha}$ and $\rho_{(\mathfrak{a}\cap\mathfrak{b})_\alpha}=\rho_{\mathfrak{a}_\alpha}\cap\rho_{\mathfrak{b}_\alpha}$. Obviously, $(\mathfrak{a}\cap\mathfrak{b})_\alpha$ are extensions of $(A\cap B)_\alpha$ for $\alpha=0,1$ (respectively), and $(\mathfrak{a}\cap\mathfrak{b})_0\cap (\mathfrak{a}\cap\mathfrak{b})_1=(0,1)$ for any nonzero $x\in V$. Let $x\in V$. Because $A,B$ are homogenous. Then
\begin{eqnarray*}
(\lambda_{(\mathfrak{a}\cap\mathfrak{b})_0}&+&\lambda_{(\mathfrak{a}\cap\mathfrak{b})_1})(x)=\sup_{x=a+b}\{\lambda_{(\mathfrak{a}\cap\mathfrak{b})_0}(a)\wedge\lambda_{(\mathfrak{a}\cap\mathfrak{b})_1}(b)\}\\
&=&\sup_{x=a+b}\{r_{(\mathfrak{a}\cap\mathfrak{b})_0}(a)e^{i2\pi\omega_{(\mathfrak{a}\cap\mathfrak{b})_0}(a)}\wedge r_{(\mathfrak{a}\cap\mathfrak{b})_1}(b)e^{i2\pi\omega_{(\mathfrak{a}\cap\mathfrak{b})_1}(b)}\}\\
&=&\sup_{x=a+b}\{(r_{\mathfrak{a}_0}(a)\wedge r_{\mathfrak{b}_0}(a))e^{i2\pi(\omega_{\mathfrak{a}_0}(a)\wedge \omega_{\mathfrak{b}_0}(a))}\wedge(r_{\mathfrak{a}_1}(b)\wedge r_{\mathfrak{b}_1}(b))e^{i2\pi(\omega_{\mathfrak{a}_1}(b)\wedge \omega_{\mathfrak{b}_1}(b))}\}\\
&=&\sup_{x=a+b}\{(r_{\mathfrak{a}_0}(a)\wedge r_{\mathfrak{a}_1}(b))e^{i2\pi(\omega_{\mathfrak{a}_0}(a)\wedge \omega_{\mathfrak{a}_1}(b))}\}\wedge\sup_{x=a+b}\{(r_{\mathfrak{b}_0}(a)\wedge r_{\mathfrak{b}_1}(b))e^{i2\pi(\omega_{\mathfrak{b}_0}(a)\wedge \omega_{\mathfrak{b}_1}(b))}\}\\
&=&r_A(x)e^{i2\pi\omega_A(x)}\wedge r_B(x)e^{i2\pi\omega_B(x)}\\
&=&\lambda_A(x)\wedge\lambda_B(x)\\
&=&\lambda_{A\cap B}(x)
\end{eqnarray*}
and
\begin{eqnarray*}
(\rho_{(\mathfrak{a}\cap\mathfrak{b})_0}&+&\rho_{(\mathfrak{a}\cap\mathfrak{b})_1})(x)=\inf_{x=a+b}\{\rho_{(\mathfrak{a}\cap\mathfrak{b})_0}(a)\vee\rho_{(\mathfrak{a}\cap\mathfrak{b})_1}(b)\}\\
&=&\inf_{x=a+b}\{\hat{r}_{(\mathfrak{a}\cap\mathfrak{b})_0}(a)e^{i2\pi\hat{\omega}_{(\mathfrak{a}\cap\mathfrak{b})_0}(a)}\vee \hat{r}_{(\mathfrak{a}\cap\mathfrak{b})_1}(b)e^{i2\pi\hat{\omega}_{(\mathfrak{a}\cap\mathfrak{b})_1}(b)}\}\\
&=&\inf_{x=a+b}\{(\hat{r}_{\mathfrak{a}_0}(a)\vee\hat{r}_{\mathfrak{b}_0}(a))e^{i2\pi(\hat{\omega}_{\mathfrak{a}_0}(a)\vee \hat{\omega}_{\mathfrak{b}_0}(a))}\vee(\hat{r}_{\mathfrak{a}_1}(b)\vee \hat{r}_{\mathfrak{b}_1}(b))e^{i2\pi(\hat{\omega}_{\mathfrak{a}_1}(b)\vee\hat{\omega}_{\mathfrak{b}_1}(b))}\}\\
&=&\inf_{x=a+b}\{(\hat{r}_{\mathfrak{a}_0}(a)\vee\hat{r}_{\mathfrak{a}_1}(b))e^{i2\pi(\hat{\omega}_{\mathfrak{a}_0}(a)\vee \hat{\omega}_{\mathfrak{a}_1}(b))}\}\vee\inf_{x=a+b}\{(\hat{r}_{\mathfrak{b}_0}(a)\vee \hat{r}_{\mathfrak{b}_1}(b))e^{i2\pi(\hat{\omega}_{\mathfrak{b}_0}(a)\vee\hat{\omega}_{\mathfrak{b}_1}(b))}\}\\
&=&\hat{r}_A(x)e^{i2\pi\hat{\omega}_A(x)}\vee\hat{r}_B(x)e^{i2\pi\hat{\omega}_B(x)}\\
&=&\rho_A(x)\vee\rho_B(x)\\
&=&\rho_{A\cap B}(x)
\end{eqnarray*}
These show that $A\cap B=(\lambda_{A\cap B},\rho_{A\cap B})$ is a CIF vector subspace of $V$. Her we show that $A\cap B=(\lambda_{A\cap B},\rho_{A\cap B})$ is a CIF lie subsuperalgebra of $V$.
Let $x,y\in V$, then
\begin{eqnarray*}
\lambda_{A\cap B}([x,y])&=&\lambda_A([x,y])\wedge\lambda_B([x,y])\\
&=&r_A([x,y])e^{i2\pi\omega_A([x,y])}\wedge r_B([x,y])e^{i2\pi\omega_B([x,y])}\\
&\geq&(r_A(x)\wedge r_A(y))e^{i2\pi(\omega_A(x)\wedge\omega_A(y))}\wedge(r_B(x)\wedge r_B(y))e^{i2\pi(\omega_B(x)\wedge\omega_B(y))}\\
&=&(r_A(x)\wedge r_B(x))e^{i2\pi(\omega_A(x)\wedge\omega_B(x))}\wedge(r_A(y)\wedge r_B(y))e^{i2\pi(\omega_A(y)\wedge\omega_B(y))}\\
&=&\lambda_{A\cap B}(x)\wedge\lambda_{A\cap B}(y)\\
\end{eqnarray*}
and
\begin{eqnarray*}
\rho_{A\cap B}([x,y])&=&\rho_A([x,y])\vee\rho_B([x,y])\\
&=&\hat{r}_A([x,y])e^{i2\pi\hat{\omega}_A([x,y])}\vee\hat{r}_B([x,y])e^{i2\pi\hat{\omega}_B([x,y])}\\
&\leq&(\hat{r}_A(x)\vee\hat{r}_A(y))e^{i2\pi(\hat{\omega}_A(x)\vee\hat{\omega}_A(y))}\vee(\hat{r}_B(x)\vee \hat{r}_B(y))e^{i2\pi(\hat{\omega}_B(x)\vee\hat{\omega}_B(y))}\\
&=&(\hat{r}_A(x)\vee\hat{r}_B(x))e^{i2\pi(\hat{\omega}_A(x)\vee\hat{\omega}_B(x))}\vee(\hat{r}_A(y)\vee \hat{r}_B(y))e^{i2\pi(\hat{\omega}_A(y)\vee\hat{\omega}_B(y))}\\
&=&\rho_{A\cap B}(x)\vee\rho_{A\cap B}(y).\\
\end{eqnarray*}
Hence $A\cap B=(\lambda_{A\cap B},\rho_{A\cap B})$ is a CIF lie subsuperalgebra of $V$.
\end{proof}
\section{On lie superalgebra anti homomoiphisms}
Remark: If $\phi : V\rightarrow V'$ is a linear map between lie superalgebras such that $\phi(a_\alpha b_\beta)=(-1)^{\alpha\beta}\phi(b_\beta)\phi(a_\alpha)$ for all $a_\alpha, b_\beta\in h(V)$, $\alpha, \beta=0,1$, then $\phi$ is called an anti-homomorphism from $V$ into $V'$. In this case for any $a_\alpha, b_\beta\in h(V)$, we have that\\
\begin{eqnarray*}
\phi([a_\alpha, b_\beta])&=&\phi(a_\alpha b_\beta-(-1)^{\alpha\beta}b_\beta a_\alpha)\\
&=&\phi(a_\alpha b_\beta)-(-1)^{\alpha\beta}\phi(b_\beta a_\alpha)\\
&=&(-1)^{\alpha\beta}\phi(b_\beta)\phi(a_\alpha)-\phi(a_\alpha)\phi(b_\beta)\\
&=&-(\phi(a_\alpha)\phi(b_\beta)-(-1)^{\alpha\beta}\phi(b_\beta)\phi(a_\alpha))\\
&=&-[\phi(a_\alpha), \phi(b_\beta)].
\end{eqnarray*}
Therefore, if $x=x_0+x_1, y=y_0+y_1\in V$, then\\
\begin{eqnarray*}
\phi([x,y])&=&\phi([x_0,y_0]+[x_0,y_1]+[x_1,y_0]+[x_1,y_1])\\
&=&\phi([x_0,y_0])+\phi([x_0,y_1])+\phi([x_1,y_0])+\phi([x_1,y_1])\\
&=&-([\phi(x_0),\phi(y_0)]+[\phi(x_0),\phi(y_1)]+[\phi(x_1),\phi(y_0)]+[\phi(x_1),\phi(y_1)])\\
&=&-[\phi(x), \phi(y)].
\end{eqnarray*}
So we have the following equivalent definition of anti-homomorphism of lie superalgebras as follows:
\begin{defn}
If $\phi: V\rightarrow V'$ is a linear map between lie superalgebras $V,\ V'$ which satisfies:
\begin{eqnarray}
\phi(V_\alpha)&\subseteq& V'_\alpha ,\  (\alpha= 0,1),\\
\phi([x, y])&=& -[\phi(x), \phi(y)]
\end{eqnarray}
Then $\phi$ is called an anti-homomorphism of lie-superalgebras.
\end{defn}
\begin{defn}
Let $A=(\lambda_A,\rho_A)$  be a CIF set of $V$. Then $A=(\lambda_A,\rho_A)$ is called an anti-complex intuitionistic fuzzy (anti-CIF for short) lie sub-superalgebra of $V$, if it satisfies the following conditions:\\
(1) $A=(\lambda_A,\rho_A)$ is a $\mathbb{Z}_2$-graded CIF vector subspace of $V$\\
(2) $\lambda_A(-[x,y])\geq\lambda_A(x)\wedge \lambda_A(y)$ and
$\rho_A(-[x,y])\leq\rho_A(x)\vee\rho_A(y)$.\\ If the condition(2) is replaced by (3) $\lambda_A(-[x,y])\geq\lambda_A(x)\vee \lambda_A(y)$ and $\rho_A(-[x,y])\leq\rho_A(x)\wedge\rho_A(y)$, then $A=(\lambda_A,\rho_A)$ is called an anti-CIF ideal of $V$.
\end{defn}
\begin{prp}
Let $\phi: V\rightarrow V'$ be an anti-homomorphism of lie-superalgebras. If $A=(\lambda_A,\rho_A)$ is an anti-CIF lie sub-superalgebra (respectively an anti-CIF ideal) of $V'$, then the CIF set $\phi^{-1}(A)$ of $V$ is also an anti-CIF lie sub-superalgebra (respectively an anti-CIF ideal).
\end{prp}
\begin{proof}
Let $x=x_0+x_1\in V$, then $\phi(x)=\phi(x_0)+\phi(x_1)\in V'$. Define $\phi^{-1}(A)_\alpha=(\lambda_{\phi^{-1}(A)_\alpha},\rho_{\phi^{-1}(A)_\alpha})$, where $\lambda_{\phi^{-1}(A)_\alpha}=\phi^{-1}(\lambda_{A_\alpha})$ and $\rho_{\phi^{-1}(A)_\alpha}=\phi^{-1}(\rho_{A_\alpha})$, $\alpha=0,1$. By Lemma 2.3 we have that they are CIF subspaces of $V_\alpha$, $\alpha=0,1$ (respectively)
Now define $\phi^{-1}(\mathfrak{a})_\alpha=(\lambda_{\phi^{-1}(\mathfrak{a})_\alpha},\rho_{\phi^{-1}(\mathfrak{a})_\alpha})$, where $\lambda_{\phi^{-1}(\mathfrak{a})_\alpha}=\phi^{-1}(\lambda_{\mathfrak{a}_\alpha})$ and $\rho_{\phi^{-1}(\mathfrak{a})_\alpha}=\phi^{-1}(\rho_{\mathfrak{a}_\alpha})$, $\alpha=0,1$. Clearly, $$\lambda_{\phi^{-1}(\mathfrak{a})_\alpha}(x)=\begin{cases} \lambda_{\phi^{-1}(A)_\alpha}(x) & :\quad x\in V_\alpha\\
0 & :\quad x\not\in V_\alpha\end{cases}\ \ and\ \
\rho_{\phi^{-1}(\mathfrak{a})_\alpha}(x)=\begin{cases} \rho_{\phi^{-1}(A)_\alpha}(x) & :\quad x\in V_\alpha\\
1 & :\quad x\not\in V_\alpha\end{cases}$$ for $\alpha=0,1$. These show that $\phi^{-1}(\mathfrak{a})_\alpha$ ($\alpha=0,1$) are CIF vector subspace of $V$. Moreover $\lambda_{\phi^{-1}(\mathfrak{a})_0}(x)\wedge \lambda_{\phi^{-1}(\mathfrak{a})_1}(x)=\phi^{-1}(\lambda_{\mathfrak{a}_0})(x)\wedge\phi^{-1}(\lambda_{\mathfrak{a}_1})(x)=\lambda_{\mathfrak{a}_0}(\phi(x))\wedge\lambda_{\mathfrak{a}_1}(\phi(x))=0$ for $0\not=x\in V$, and similarly, $\rho_{\phi^{-1}(\mathfrak{a})_0}(x)\vee \rho_{\phi^{-1}(\mathfrak{a})_1}(x)=1$ for $0\not=x\in V$.\\
Let $0\not=x\in V$. Then
\begin{eqnarray*}
\lambda_{\phi^{-1}(\mathfrak{a})_0+\phi^{-1}(\mathfrak{a})_1}(x)&=&\sup_{x=a+b}\{\lambda_{\phi^{-1}(\mathfrak{a})_0}(a)\wedge\lambda_{\phi^{-1}(\mathfrak{a})_1}(b)\}\\
&=&\sup_{x=a+b}\{\lambda_{\mathfrak{a}_0}(\phi(a))\wedge\lambda_{\mathfrak{a}_1}(\phi(b))\},\ (a=a_0+a_1,\ b=b_0+b_1)\\
&=&\sup_{x=a+b}\{\lambda_{\mathfrak{a}_0}(\phi(a_0))\wedge\lambda_{\mathfrak{a}_1}(\phi(b_1))\}\\
&=&\sup_{x=a+b}\{r_{\mathfrak{a}_0}(\phi(a_0))e^{i2\pi\omega_{\mathfrak{a}_0}(\phi(a_0))}\wedge r_{\mathfrak{a}_1}(\phi(b_1))e^{i2\pi\omega_{\mathfrak{a}_1}(\phi(b_1))}\}\\
&=&r_{\mathfrak{a}_0}(\phi(x_0))e^{i2\pi\omega_{\mathfrak{a}_0}(\phi(x_0))}\wedge r_{\mathfrak{a}_1}(\phi(x_1))e^{i2\pi\omega_{\mathfrak{a}_1}(\phi(x_1))}\\
&=&\lambda_{\mathfrak{a}_0}(\phi(x_0))\wedge\lambda_{\mathfrak{a}_1}(\phi(x_1))\\
&=&\lambda_{\mathfrak{a}_0+\mathfrak{a}_1}(\phi(x)),\ (x=x_0+x_1)\\
&=&\lambda_A(\phi(x))\\
&=&\lambda_{\phi^{-1}(A)}(x).
\end{eqnarray*}
and
\begin{eqnarray*}
\rho_{\phi^{-1}(\mathfrak{a})_0+\phi^{-1}(\mathfrak{a})_1}(x)&=&\inf_{x=a+b}\{\rho_{\phi^{-1}(\mathfrak{a})_0}(a)\vee\rho_{\phi^{-1}(\mathfrak{a})_1}(b)\}\\
&=&\inf_{x=a+b}\{\rho_{\mathfrak{a}_0}(\phi(a))\vee\rho_{\mathfrak{a}_1}(\phi(b))\},\ (a=a_0+a_1,\ b=b_0+b_1)\\
&=&\inf_{x=a+b}\{\rho_{\mathfrak{a}_0}(\phi(a_0))\vee\rho_{\mathfrak{a}_1}(\phi(b_1))\}\\
&=&\inf_{x=a+b}\{\hat{r}_{\mathfrak{a}_0}(\phi(a_0))e^{i2\pi\hat{\omega}_{\mathfrak{a}_0}(\phi(a_0))}\vee \hat{r}_{\mathfrak{a}_1}(\phi(b_1))e^{i2\pi\hat{\omega}_{\mathfrak{a}_1}(\phi(b_1))}\}\\
&=&\hat{r}_{\mathfrak{a}_0}(\phi(x_0))e^{i2\pi\hat{\omega}_{\mathfrak{a}_0}(\phi(x_0))}\vee \hat{r}_{\mathfrak{a}_1}(\phi(x_1))e^{i2\pi\hat{\omega}_{\mathfrak{a}_1}(\phi(x_1))}\\
&=&\rho_{\mathfrak{a}_0}(\phi(x_0))\vee\rho_{\mathfrak{a}_1}(\phi(x_1))\\
&=&\rho_{\mathfrak{a}_0+\mathfrak{a}_1}(\phi(x)),\ (x=x_0+x_1)\\
&=&\rho_A(\phi(x))\\
&=&\rho_{\phi^{-1}(A)}(x).
\end{eqnarray*}
So, $\phi^{-1}(A)=\phi^{-1}(A)_0\oplus\phi^{-1}(A)_1$ is a $\mathbb{Z}_2$-graded CIF vector subspace of $V$.\\ Let $x,y\in V$. Then
\begin{eqnarray*}
\lambda_{\phi^{-1}(A)}(-[x,y])=\lambda_A(-\phi([x,y]))
&=&\lambda_A([\phi(x),\phi(y)])\\
&\geq&\lambda_A(\phi(x))\wedge\lambda_A(\phi(y))\\
&=&\lambda_{\phi^{-1}(A)}(x)\wedge\lambda_{\phi^{-1}(A)}(y)
\end{eqnarray*}
and
\begin{eqnarray*}
\rho_{\phi^{-1}(A)}(-[x,y])=\rho_A(-\phi([x,y]))
&=&\rho_A([\phi(x),\phi(y)])\\
&\leq&\rho_A(\phi(x))\vee\rho_A(\phi(y))\\
&=&\rho_{\phi^{-1}(A)}(x)\vee\rho_{\phi^{-1}(A)}(y).
\end{eqnarray*}
Thus, $\phi^{-1}(A)$ is an anti-CIF Lie subsuperalgebra of $V$.\\
Also if $x,y\in V$. Then
\begin{eqnarray*}
\lambda_{\phi^{-1}(A)}(-[x,y])=\lambda_A(-\phi([x,y]))
&=&\lambda_A([\phi(x),\phi(y)])\\
&\geq&\lambda_A(\phi(x))\vee\lambda_A(\phi(y))\\
&=&\lambda_{\phi^{-1}(A)}(x)\vee\lambda_{\phi^{-1}(A)}(y)
\end{eqnarray*}
and
\begin{eqnarray*}
\rho_{\phi^{-1}(A)}(-[x,y])=\rho_A(-\phi([x,y]))
&=&\rho_A([\phi(x),\phi(y)])\\
&\leq&\rho_A(\phi(x))\wedge\rho_A(\phi(y))\\
&=&\rho_{\phi^{-1}(A)}(x)\wedge\rho_{\phi^{-1}(A)}(y).
\end{eqnarray*}
Thus, $\phi^{-1}(A)$ is an anti-CIF ideal of $V$.
\end{proof}
\begin{prp}\label{sec4-1}
Let $\phi: V\rightarrow V'$ be a surjective anti-homomorphism of lie-superalgebras. If $A=(\lambda_A,\rho_A)$ is an anti-CIF lie sub-superalgebra (respectively an anti-CIF ideal) of $V$, then the CIF set $\phi(A)$ of $V'$ is also an anti-CIF lie sub-superalgebra (respectively an anti-CIF ideal).
\end{prp}
\begin{proof}
It is easy to see that $A=A_0\oplus A_1$ where
$A_0=(\lambda_{A_0},\rho_{A_0})$, $A_1=(\lambda_{A_1},\rho_{A_1})$
are CIF vector subspaces of $V_0, V_1$ (respectively) since
$A=(\lambda_A,\rho_A)$ is a CIF lie sub-superalgebra of $V$. Define
$\phi(A)_\alpha=(\lambda_{\phi(A)_\alpha},\rho_{\phi(A)_\alpha})$,
where $\lambda_{\phi(A)_\alpha}=\phi(\lambda_{A_\alpha})$,
$\rho_{\phi(A)_\alpha}=\phi(\rho_{A_\alpha})$ for $\alpha=0,1$. By
Lemma 2.4, $\phi(A)_\alpha$ is an anti-CIF subspace of $V_\alpha$,
$(\alpha=0,1)$. Extend them to $\phi(A)^{'}_\alpha$ $(\alpha=0,1)$
(respectively), we define
$\phi(\mathfrak{a})_\alpha=(\lambda_{\phi(\mathfrak{a})_\alpha},\rho_{\phi(\mathfrak{a})_\alpha})$
where
$\lambda_{\phi(\mathfrak{a})_\alpha}=\phi(\lambda_{\mathfrak{a}_\alpha})$,
$\rho_{\phi(\mathfrak{a})_\alpha}=\phi(\rho_{\mathfrak{a}_\alpha})$
for $(\alpha=0,1)$ (respectively). Clearly
$$\lambda_{\phi(\mathfrak{a})_\alpha}(x)=\begin{cases} \lambda_{\phi(A)_\alpha}(x) & :\quad x\in V'_\alpha\\
0 & :\quad x\not\in V'_\alpha\end{cases}\ \ and\ \
\rho_{\phi(\mathfrak{a})_\alpha}(x)=\begin{cases} \rho_{\phi(A)_\alpha}(x) & :\quad x\in V'_\alpha\\
1 & :\quad x\not\in V'_\alpha\end{cases}$$ for $\alpha=0,1$.\\
If $0\not=x\in V^{'}$, then
\begin{eqnarray*}
\lambda_{\phi(\mathfrak{a})_0}(x)\wedge\lambda_{\phi(\mathfrak{a})_1}(x)&=&\phi(\lambda_{\mathfrak{a}_0})(x)\wedge\phi(\lambda_{\mathfrak{a}_1})(x)\\
&=&\sup_{x=\phi(a)}\{\lambda_{\mathfrak{a}_0}(a)\}\wedge\sup_{x=\phi(a)}\{\lambda_{\mathfrak{a}_1(a)}\}\\
&=&\sup_{x=\phi(a)}\{\lambda_{\mathfrak{a}_0}(a)\wedge\lambda_{\mathfrak{a}_1(a)}\}=0\\
\end{eqnarray*}
and
\begin{eqnarray*}
\rho_{\phi(\mathfrak{a})_0}(x)\vee\rho_{\phi(\mathfrak{a})_1}(x)&=&\phi(\rho_{\mathfrak{a}_0})(x)\vee\phi(\rho_{\mathfrak{a}_1})(x)\\
&=&\inf_{x=\phi(a)}\{\rho_{\mathfrak{a}_0}(a)\}\vee\inf_{x=\phi(a)}\{\rho_{\mathfrak{a}_1(a)}\}\\
&=&\inf_{x=\phi(a)}\{\rho_{\mathfrak{a}_0}(a)\vee\rho_{\mathfrak{a}_1(a)}\}=1.\\
\end{eqnarray*}
Let $0\not=y\in V^{'}$. Then
\begin{eqnarray*}
\lambda_{\phi(\mathfrak{a})_0+\phi(\mathfrak{a})_1}(y)&=&\sup_{y=a+b}\{\lambda_{\phi(\mathfrak{a})_0}(a)\wedge\lambda_{\phi(\mathfrak{a})_1(b)}\}\\
&=&\sup_{y=a+b}\{\phi(\lambda_{\mathfrak{a}_0})(a)\wedge\phi(\lambda_{\mathfrak{a}_1})(b)\}\\
&=&\sup_{y=a+b}\{\sup_{a=\phi(m)}\{\lambda_{\mathfrak{a}_0}(m)\}\wedge\sup_{b=\phi(n)}\{\lambda_{\mathfrak{a}_1}(n)\}\}\\
&=&\sup_{y=\phi(x)}\{\sup_{x=m+n}\{\lambda_{\mathfrak{a}_0}(m)\wedge\lambda_{\mathfrak{a}_1}(n)\}\}\\
&=&\sup_{y=\phi(x)}\{\sup_{x=m+n}\{r_{\mathfrak{a}_0}(m)e^{i2\pi\omega_{\mathfrak{a}_0}(m)}\wedge r_{\mathfrak{a}_1}(n)e^{i2\pi\omega_{\mathfrak{a}_1}(n)}\}\}\\
&=&\sup_{y=\phi(x)}\{\sup_{x=m+n}\{r_{\mathfrak{a}_0}(m)\wedge r_{\mathfrak{a}_1}(n)\}e^{i2\pi\sup\limits_{x=m+n}\{\omega_{\mathfrak{a}_0}(m)\wedge \omega_{\mathfrak{a}_1}(n)\}}\}\\
&=&\sup_{y=\phi(x)}\{r_{\mathfrak{a}_0+\mathfrak{a}_1}(x)e^{i2\pi\omega_{\mathfrak{a}_0+\mathfrak{a}_1}(x)}\}\\
&=&\sup_{y=\phi(x)}\{\lambda_{\mathfrak{a}_0+\mathfrak{a}_1}(x)\}=\sup_{y=\phi(x)}\{\lambda_A(x)\}=\lambda_{\phi(A)}(y).
\end{eqnarray*}
and
\begin{eqnarray*}
\rho_{\phi(\mathfrak{a})_0+\phi(\mathfrak{a})_1}(y)&=&\inf_{y=a+b}\{\rho_{\phi(\mathfrak{a})_0}(a)\vee\rho_{\phi(\mathfrak{a})_1(b)}\}\\
&=&\inf_{y=a+b}\{\phi(\rho_{\mathfrak{a}_0})(a)\vee\phi(\rho_{\mathfrak{a}_1})(b)\}\\
&=&\inf_{y=a+b}\{\inf_{a=\phi(m)}\{\rho_{\mathfrak{a}_0}(m)\}\vee\inf_{b=\phi(n)}\{\rho_{\mathfrak{a}_1}(n)\}\}\\
&=&\inf_{y=\phi(x)}\{\inf_{x=m+n}\{\rho_{\mathfrak{a}_0}(m)\vee\rho_{\mathfrak{a}_1}(n)\}\}\\
&=&\inf_{y=\phi(x)}\{\inf_{x=m+n}\{\hat{r}_{\mathfrak{a}_0}(m)e^{i2\pi\hat{\omega}_{\mathfrak{a}_0}(m)}\vee \hat{r}_{\mathfrak{a}_1}(n)e^{i2\pi\hat{\omega}_{\mathfrak{a}_1}(n)}\}\}\\
&=&\inf_{y=\phi(x)}\{\inf_{x=m+n}\{\hat{r}_{\mathfrak{a}_0}(m)\vee \hat{r}_{\mathfrak{a}_1}(n)\}e^{i2\pi\inf\limits_{x=m+n}\{\hat{\omega}_{\mathfrak{a}_0}(m)\vee \hat{\omega}_{\mathfrak{a}_1}(n)\}}\}\\
&=&\inf_{y=\phi(x)}\{\hat{r}_{\mathfrak{a}_0+\mathfrak{a}_1}(x)e^{i2\pi\hat{\omega}_{\mathfrak{a}_0+\mathfrak{a}_1}(x)}\}\\
&=&\inf_{y=\phi(x)}\{\rho_{\mathfrak{a}_0+\mathfrak{a}_1}(x)\}=\inf_{y=\phi(x)}\{\rho_A(x)\}=\rho_{\phi(A)}(y).
\end{eqnarray*}
So $\phi(A)=\phi(A)_0\oplus\phi(A)_1$ is a $\mathbb{Z}_2$ graded CIF vector subspace of $V^{'}$.\\
Let $x,y\in V^{'}$. We need to show that $\lambda_{\phi(A)}(-[x,y])\geq\lambda_{\phi(A)}(x)\wedge\lambda_{\phi(A)}(y)$ and $\rho_{\phi(A)}(-[x,y])\leq\rho_{\phi(A)}(x)\vee\rho_{\phi(A)}(y)$. Suppose that $\lambda_{\phi(A)}(-[x,y])<\lambda_{\phi(A)}(x)\wedge\lambda_{\phi(A)}(y)=r_{\phi(A)}(x)e^{i2\pi\omega_{\phi(A)}(x)}\wedge r_{\phi(A)}(y)e^{i2\pi\omega_{\phi(A)}(y)}$. Then $\lambda_{\phi(A)}(-[x,y])=r_{\phi(A)}(-[x,y])e^{i2\pi\omega_{\phi(A)}(-[x,y])}<\{r_{\phi(A)}(x)\wedge r_{\phi(A)}(y)\}e^{i2\pi\{\omega_{\phi(A)}(x)\wedge \omega_{\phi(A)}(y)\}}$, since $\phi(A)$ is homogeneous. Thus, $r_{\phi(A)}(-[x,y])<r_{\phi(A)}(x)\wedge r_{\phi(A)}(y)$ or $\omega_{\phi(A)}(-[x,y])<\omega_{\phi(A)}(x)\wedge\omega_{\phi(A)}(y)$. If $r_{\phi(A)}(-[x,y])<r_{\phi(A)}(x)\wedge r_{\phi(A)}(y)$, then $r_{\phi(A)}(-[x,y])<r_{\phi(A)}(x)$ and $r_{\phi(A)}(-[x,y])<r_{\phi(A)}(y)$. Let $t\in[0,1]$ such that $r_{\phi(A)}(-[x,y])<t<r_{\phi(A)}(x)$ and $r_{\phi(A)}(-[x,y])<t<r_{\phi(A)}(y)$. Then there exist $a\in\phi^{-1}(x)$ and $b\in\phi^{-1}(y)$ such that $r_A(a)>t$, $r_A(b)>t$. Since $\phi([a,b])=-[x,y]$, we have that
\begin{eqnarray*}
r_{\phi(A)}(-[x,y])=\sup_{-[x,y]=\phi([a^{'},b^{'}])}\{r_A([a^{'},b^{'}])\}&\geq& r_A([a,b])\\
&\geq& r_A(a)\wedge r_A(b)\\
&>& t> r_{\phi(A)}(-[x,y])
\end{eqnarray*}
which is a contradiction. Similarly for the case $\omega_{\phi(A)}(-[x,y])<\omega_{\phi(A)}(x)\wedge \omega_{\phi(A)}(y)$. Therefore, $\phi(A)$ is an anti-CIF lie sub-superalgebra of $V^{'}$.\\
Now to show that $\phi(A)$ is an anti-CIF ideal of $V^{'}$, we need to prove that for any $x,y\in V^{'}$, then $\lambda_{\phi(A)}(-[x,y])\geq\lambda_{\phi(A)}(x)\vee\lambda_{\phi(A)}(y)$ and $\rho_{\phi(A)}(-[x,y])\leq\rho_{\phi(A)}(x)\wedge\rho_{\phi(A)}(y)$. Suppose that $\lambda_{\phi(A)}(-[x,y])<\lambda_{\phi(A)}(x)\vee\lambda_{\phi(A)}(y)=r_{\phi(A)}(x)e^{i2\pi\omega_{\phi(A)}(x)}\vee r_{\phi(A)}(y)e^{i2\pi\omega_{\phi(A)}(y)}$. Then, since $\phi(A)$ is homogeneous,
$\lambda_{\phi(A)}(-[x,y])=r_{\phi(A)}(-[x,y])e^{i2\pi\omega_{\phi(A)}(-[x,y])}<\{r_{\phi(A)}(x)\vee r_{\phi(A)}(y)\}e^{i2\pi\{\omega_{\phi(A)}(x)\vee\omega_{\phi(A)}(y)\}}$. Which implies that $r_{\phi(A)}(-[x,y])$ $<r_{\phi(A)}(x)\vee r_{\phi(A)}(y)$ or $\omega_{\phi(A)}(-[x,y])<\omega_{\phi(A)}(x)\vee\omega_{\phi(A)}(y)$. If $r_{\phi(A)}(-[x,y])<r_{\phi(A)}(x)\vee r_{\phi(A)}(y)$, then $r_{\phi(A)}(-[x,y])<r_{\phi(A)}(x)$ or $r_{\phi(A)}(-[x,y])<r_{\phi(A)}(y)$. Suppose that $r_{\phi(A)}(-[x,y])<r_{\phi(A)}(x)$, then choose $t\in[0,1]$ such that $r_{\phi(A)}(-[x,y])<t<r_{\phi(A)}(x)$, so there exists $a\in\phi^{-1}(x)$ such that $r_A(a)>t$ and since $\phi$ is onto there exists $b\in\phi^{-1}(y)$. Now, since $\phi([a,b])=-[\phi(a),\phi(b)]=-[x,y]$ then we have that
\begin{eqnarray*}
r_{\phi(A)}(-[x,y])&=&\sup_{-[x,y]=\phi([a^{'},b^{'}])}\{r_A([a^{'},b^{'}])\}\\
&=&\sup_{-[x,y]=\phi([a^{'},b^{'}])}\{r_A(a^{'})\vee r_A(b^{'})\}\\
&\geq& r_A(a)\vee r_A(b)\\
&>& t> r_{\phi(A)}(-[x,y])
\end{eqnarray*}
which is a contradiction. The other case can be proved similarly. Also, suppose that $\rho_{\phi(A)}(-[x,y])>\rho_{\phi(A)}(x)\wedge\rho_{\phi(A)}(y)$. Then $\hat{r}_{\phi(A)}(-[x,y])e^{i2\pi\hat{\omega}_{\phi(A)}(-[x,y])}>\{\hat{r}_{\phi(A)}(x)\wedge \hat{r}_{\phi(A)}(y)\}e^{i2\pi\{\hat{\omega}_{\phi(A)}(x)\wedge\hat{\omega}_{\phi(A)}(y)\}}$, since $\phi(A)$ is homogenous we have that $\hat{r}_{\phi(A)}(-[x,y])>\hat{r}_{\phi(A)}(x)\wedge \hat{r}_{\phi(A)}(y)$ or $\hat{\omega}_{\phi(A)}(-[x,y])>\hat{\omega}_{\phi(A)}(x)\wedge\hat{\omega}_{\phi(A)}(y)$. If $\hat{r}_{\phi(A)}(-[x,y])>\hat{r}_{\phi(A)}(x)\wedge \hat{r}_{\phi(A)}(y)$, then $\hat{r}_{\phi(A)}(-[x,y])>\hat{r}_{\phi(A)}(x)$ or $\hat{r}_{\phi(A)}(-[x,y])>\hat{r}_{\phi(A)}(y)$. Without loss of generality we may assume that $\hat{r}_{\phi(A)}(-[x,y])>\hat{r}_{\phi(A)}(x)$, choose $t\in[0,1]$ such that $\hat{r}_{\phi(A)}(-[x,y])>t>\hat{r}_{\phi(A)}(x)$, then there exists $a\in\phi^{-1}(x)$ such that $\hat{r}_A(a)<t$, because $\phi$ is onto map let $b\in\phi^{-1}(y)$. Since $\phi([a,b])=-[x,y]$, we have that
\begin{eqnarray*}
\hat{r}_{\phi(A)}(-[x,y])&=&\inf_{-[x,y]=\phi([a^{'},b^{'}])}\{\hat{r}_A([a^{'},b^{'}])\}\\
&=&\inf_{-[x,y]=\phi([a^{'},b^{'}])}\{\hat{r}_A(a^{'})\wedge \hat{r}_A(b^{'})\}\\
&\leq& \hat{r}_A(a)\wedge \hat{r}_A(b)\\
&<& t< \hat{r}_{\phi(A)}(-[x,y])
\end{eqnarray*}
which is a contradiction. The other case $\hat{\omega}_{\phi(A)}(-[x,y])>\hat{\omega}_{\phi(A)}(x)\wedge\hat{\omega}_{\phi(A)}(y)$ can be proved similarly. Therefore, $\phi(A)$ is an anti-CIF ideal of $V^{'}$.
\end{proof}
\begin{Theorem}
Let $\phi: V\rightarrow V'$ be a surjective anti-homomorphism of lie-superalgebras. If $A=(\lambda_A,\rho_A)$ and  $B=(\lambda_B,\rho_B)$ are anti-CIF ideals of $V$, then the CIF set $\phi(A+B)=\phi(A)+\phi(B)$ of $V'$ is also an anti-CIF ideal.
\end{Theorem}
\begin{proof}
We already proved in Proposition~\ref{sec4-1} that $\phi(A+B)$ is an anti-CIF ideal of $V^{'}$. Therefore the only thing we need to prove is that  $\phi(A+B)=\phi(A)+\phi(B)$. Let $y\in V^{'}$, then
\begin{eqnarray*}
\lambda_{\phi(A+B)}(y)&=&\sup_{y=\phi(x)}\{\lambda_{A+B}(x)\}\\
&=&\sup_{y=\phi(x)}\{\sup_{x=a+b}\{\lambda_A(a)\wedge\lambda_B(b)\}\}\\
&=&\sup_{y=\phi(x)}\{\sup_{x=a+b}\{r_A(a)e^{i2\pi\omega_A(a)}\wedge r_B(b)e^{i2\pi\omega_B(b)} \}\}\\
&=&\sup_{y=\phi(x)}\{\sup_{x=a+b}\{(r_A(a)\wedge r_B(b))e^{i2\pi(\omega_A(a)\wedge\omega_B(b))}\}\}\\
&=&\sup_{y=\phi(a)+\phi(b)}\{(r_A(a)\wedge r_B(b))e^{i2\pi(\omega_A(a)\wedge\omega_B(b))}\}\ \ (A\ and\ B\ are\ homogenous)\\
&=&\sup_{y=m+n}\{(\sup_{m=\phi(a)}\{r_A(a)\}\wedge\sup_{n=\phi(b)}\{r_B(b)\})e^{i2\pi(\sup\limits_{m=\phi(a)}\{\omega_A(a)\}\wedge\sup\limits_{n=\phi(b)}\{\omega_B(b)\})}\}\\
&=&\sup_{y=m+n}\{(r_{\phi(A)}(m)\wedge r_{\phi(B)}(n))e^{i2\pi(\omega_{\phi(A)}(m)\wedge \omega_{\phi(B)}(n))}\}\\
&=&\sup_{y=m+n}\{r_{\phi(A)}(m)e^{i2\pi\omega_{\phi(A)}(m)}\wedge r_{\phi(B)}(n)e^{i2\pi\omega_{\phi(B)}(n)}\}\\
&=&\sup_{y=m+n}\{\lambda_{\phi(A)}(m)\wedge\lambda_{\phi(B)}(n)\}\\
&=&\lambda_{\phi(A)+\phi(B)}(y)
\end{eqnarray*}
and
\begin{eqnarray*}
\rho_{\phi(A+B)}(y)&=&\inf_{y=\phi(x)}\{\rho_{A+B}(x)\}\\
&=&\inf_{y=\phi(x)}\{\inf_{x=a+b}\{\rho_A(a)\vee\rho_B(b)\}\}\\
&=&\inf_{y=\phi(x)}\{\inf_{x=a+b}\{\hat{r}_A(a)e^{i2\pi\hat{\omega}_A(a)}\vee \hat{r}_B(b)e^{i2\pi\hat{\omega}_B(b)} \}\}\\
&=&\inf_{y=\phi(x)}\{\inf_{x=a+b}\{(\hat{r}_A(a)\vee \hat{r}_B(b))e^{i2\pi(\hat{\omega}_A(a)\vee\hat{\omega}_B(b))}\}\}\\
&=&\inf_{y=\phi(a)+\phi(b)}\{(\hat{r}_A(a)\vee\hat{r}_B(b))e^{i2\pi(\hat{\omega}_A(a)\vee\hat{\omega}_B(b))}\}\ \ (A\ and\ B\ are\ homogenous)\\
&=&\inf_{y=m+n}\{(\inf_{m=\phi(a)}\{\hat{r}_A(a)\}\vee\inf_{n=\phi(b)}\{\hat{r}_B(b)\})e^{i2\pi(\inf\limits_{m=\phi(a)}\{\hat{\omega}_A(a)\}\vee\inf\limits_{n=\phi(b)}\{\hat{\omega}_B(b)\})}\}\\
&=&\inf_{y=m+n}\{(\hat{r}_{\phi(A)}(m)\vee\hat{r}_{\phi(B)}(n))e^{i2\pi(\hat{\omega}_{\phi(A)}(m)\vee \hat{\omega}_{\phi(B)}(n))}\}\\
&=&\inf_{y=m+n}\{\hat{r}_{\phi(A)}(m)e^{i2\pi\hat{\omega}_{\phi(A)}(m)}\vee \hat{r}_{\phi(B)}(n)e^{i2\pi\hat{\omega}_{\phi(B)}(n)}\}\\
&=&\inf_{y=m+n}\{\rho_{\phi(A)}(m)\vee\rho_{\phi(B)}(n)\}\\
&=&\rho_{\phi(A)+\phi(B)}(y).
\end{eqnarray*}
So, $\phi(A+B)=\phi(A)+\phi(B)$ is an anti-CIF ideal of $V^{'}$.
\end{proof}
\section{Conclusion}
In this article, we define the complex intuitionistic fuzzy Lie sub-superalgebras and complex intuitionistic fuzzy ideals of Lie superalgebras. Then, we study some related properties of complex intuitionistic fuzzy Lie sub-superalgebras and complex intuitionistic fuzzy ideals. Finally, we define the image and preimage of complex intuitionistic fuzzy Lie sub-superalgebra under Lie superalgebra anti-homomorphism. The properties of anti-complex intuitionistic fuzzy Lie sub-superalgebras and anti-complex intuitionistic fuzzy ideals under anti-homomorphisms of Lie superalgebras are investigated.

\end{document}